\documentclass[letterpaper, 10 pt, conference]{ieeeconf}  

\IEEEoverridecommandlockouts                              

\pdfoutput = 1
\usepackage{cite}
\usepackage{amsmath,amssymb,amsfonts}
\usepackage{algorithmic}
\usepackage{textcomp,float}
\usepackage{graphicx,subfigure}
\usepackage{epsfig} 
\usepackage{mathptmx} 
\usepackage{times} 
\usepackage{amsmath} 
\usepackage{amssymb}  
\usepackage{amsfonts}
\usepackage{euscript}
\usepackage{color}
\usepackage{ccaption}
\usepackage{listings}
\usepackage{booktabs} 
\usepackage{longtable}
\usepackage{multirow}
\usepackage{graphicx}
\usepackage{epstopdf}
\usepackage{makecell}

\newtheorem{theorem}{\bf \text{Theorem}}

\newtheorem{assumption}{\bf \text{Assumption}}

\newtheorem{corollary}{\bf \text{Corollary}}
\newtheorem{lemma}{\bf \text{Lemma}}
\newtheorem{remark}{\bf \text{Remark}}

\newcommand{\E}{{\mathbb E}}

\newcommand{\Prob}{{\mathbb P}}
\newcommand{\R}{{\mathbb R}}
\newcommand{\0}{\bf{0}}
\newcommand{\TR}{\text{R}}

\newcommand{\Var}{\text{Var}}

\newcommand{\EB}{\text{EB}}
\newcommand{\Sy}{\text{Sy}}
\newcommand{\EEB}{\text{EEB}}
\newcommand{\SURE}{\text{SURE}}

\makeatletter

\newcommand{\Rmnum}[1]{\expandafter \@slowromancap\romannumeral #1@}
\makeatother

\DeclareMathOperator*{\rank}{rank}

\DeclareMathOperator*{\cond}{cond}
\DeclareMathOperator*{\Tr}{Tr}
\DeclareMathOperator*{\plim}{plim}
\DeclareMathOperator*{\MSE}{MSE}
\DeclareMathOperator*{\LS}{LS}

\DeclareMathOperator*\argmin{arg\,min}

\title{\LARGE \bf
Supplementary Material for CDC Submission No. 1461*
}

\author{Yue Ju$^{1}$, Tianshi Chen$^{2}$, Biqiang Mu$^{3}$ and Lennart Ljung$^{4}$
\thanks{*This work was supported by the Thousand Youth Talents
Plan funded by the central government of China, the general project funded by NSFC under contract No. 61773329, the Shenzhen research projects funded by the Shenzhen Science and Technology Innovation Council under contract No. Ji-20170189, the President's grant under contract No. PF. 01.000249 and the Start-up grant under contract No. 2014.0003.23 funded by CUHKSZ.}
\thanks{$^{1}$Yue Ju is with the Chinese University of Hong Kong, Shenzhen 518172, China {\tt\small yueju@link.cuhk.edu.cn}}%
\thanks{$^{2}$Tianshi Chen is with the School of Science and Engineering and Shenzhen Research Institute of Big Data, The Chinese University of Hong Kong, Shenzhen 518172, China {\tt\small tschen@cuhk.edu.cn}}%
\thanks{$^{3}$Biqiang Mu is with Key Laboratory of Systems and Control, Institute of Systems Science, Academy of Mathematics and System Science, Chinese Academy of Sciences, Beijing 100190, China {\tt\small bqmu@amss.ac.cn}}
\thanks{$^{4}$Lennart Ljung is with the Division of Automatic Control, Department of Electrical Engineering, Link$\ddot{\text{o}}$ping University, Link$\ddot{\text{o}}$ping SE-58183, Sweden {\tt\small ljung@isy.liu.se}}
}

\begin{document}

\maketitle
\thispagestyle{empty}
\pagestyle{empty}

\begin{abstract}

In this paper, we study the influence of ill-conditioned regression matrix on
two hyper-parameter estimation methods for the kernel-based regularization method: the empirical Bayes ($\EB$) and the Stein's unbiased risk estimator (SURE). First, we consider the convergence rate of the cost functions of EB and SURE, and we find that
they have the same convergence rate but the influence of the ill-conditioned regression matrix on the scale factor are different: the scale factor for SURE
contains one more factor $\cond(\Phi^{T}\Phi)$
than that of EB, where $\Phi$ is the regression matrix and $\cond(\cdot)$ denotes the condition number of a matrix. This finding indicates that when $\Phi$ is ill-conditioned, i.e., $\cond(\Phi^{T}\Phi)$ is large, the cost function of SURE converges slower than that of EB. Then we consider the convergence rate of the optimal hyper-parameters of EB and SURE, and we find that they are both asymptotically normally distributed and have the same convergence rate, but the influence of the ill-conditioned regression matrix on the scale factor are different. In particular, for the ridge regression case, we show that the optimal hyper-parameter of SURE converges slower than that of EB with a factor of $1/n^2$, as $\cond(\Phi^{T}\Phi)$ goes to $\infty$,
where $n$ is the FIR model order.
 \end{abstract}

\section{INTRODUCTION}

In the past decade, kernel-based regularization method (KRM) has gained increasing attention and become a hot research topic in the system identification community.
It has been shown to be a complement \cite{PDCDL2014} to the classical system identification paradigm based on top of the maximum likelihood/ prediction error methods (ML/PEM) \cite{Ljung1999}.

There are two fundamental issues for KRM. The first one is the parameterization of kernel structure with hyper-parameters based on the prior knowledge of the system to be identified, for which several kernels have been invented, e.g. \cite{PD2010} and \cite{COL2012}. The other one is the tuning of hyper-parameters based on the given data to achieve balance in the bias-variance trade-off. Common methods for the hyper-parameter estimation include the cross-validation (CV), empirical Bayes (EB), $C_{\text{p}}$ statistics, Stein's unbiased estimator ($\SURE$) and so on. Among them, $\SURE$ has two variants: $\SURE_{\text{g}}$ corresponds to the mean square error (MSE) related with the impulse response reconstruction ($\MSE_{g}$), while $\SURE_{\text{y}}$ corresponds to the MSE with respect to output prediction ($\MSE_{y}$) \cite{PC2015,MCL2018}.

An interesting phenomena we observed from Monte Carlo simulations in \cite{PDCDL2014,PC2015,MCL2018} is that when the regression matrix is ill-conditioned, $\SURE_{\text{g}}$ and $\SURE_{\text{y}}$ often converges slower than EB, which motivates us to consider the following two questions:
\begin{enumerate}
\item what is the impact of the condition number of $\Phi^{T}\Phi$ on the convergence properties of the cost functions of $\EB$ and $\SURE_{\text{y}}$?

\item what is the impact of the condition number of $\Phi^{T}\Phi$ on the convergence properties of the optimal hyper-parameters of $\EB$ and $\SURE_{\text{y}}$?
\end{enumerate}

In this paper, we use asymptotic analysis tools from statistics to address these questions. In particular, we first consider the convergence rate of the cost functions of EB and SURE, and we find that
they have the same convergence rate but the influence of the ill-conditioned regression matrix on the scale factor are different: the scale factor for SURE
contains one more factor $\cond(\Phi^{T}\Phi)$
than that of EB, where $\Phi$ is the regression matrix and $\cond(\cdot)$ denotes the condition number of a matrix. This finding indicates that when $\Phi$ is ill-conditioned, i.e., $\cond(\Phi^{T}\Phi)$ is large, the cost function of SURE converges slower than that of EB. Then we consider the convergence rate of the optimal hyper-parameters of EB and SURE, and we find that they are both asymptotically normally distributed and have the same convergence rate, but the influence of the ill-conditioned regression matrix on the scale factor are different. In particular, for the ridge regression case, we show that the optimal hyper-parameter of SURE converges slower than that of EB with a factor of $1/n^2$, as $\cond(\Phi^{T}\Phi)$ goes to $\infty$,
where $n$ is the FIR model order.


The remaining parts of the paper is organized as follows. In Section $\text{\Rmnum2}$, we introduce the LS method and the RLS method for the linear regression model. In Section $\text{\Rmnum3}$, we show several common kernel structures and hyper-parameter estimation methods, including $\EB$ and $\SURE_{\text{y}}$. In Section $\text{\Rmnum4}$, we compute the upper bounds of the convergence rates of $\EB$ and $\SURE_{\text{y}}$ cost functions and compare them in terms of the greatest power of the condition number of $\Phi^{T}\Phi$. In Section $\text{\Rmnum5}$, we derive the asymptotic normality of $\EB$ and $\SURE_{\text{y}}$ hyper-parameter estimators. In Section $\text{\Rmnum6}$, we illustrate our experiment results with the Monte Carlo simulation method. Our conclusion is given in Section $\text{\Rmnum7}$. All proofs of the theorems and corollaries are listed in Appendix.

\section{Regularized Least Squares Estimation for the Linear Regression Model}

We focus on the linear regression model:
\begin{align}\label{eq:liner regression model at time t}
y(t)=\phi^{T}(t)\theta+v(t),\ t=1,\cdots,N,
\end{align}
where $t$ denotes the time index, $y(t)\in\R$, $\phi(t)\in\R^{n}$, $v(t)\in\R$ represent the output, the regressor and the disturbance at time $t$, and $\theta\in\R^{n}$ is the unknown parameter to be estimated. In addition, $v(t)$ is assumed to be independent and identically distributed (i.i.d.) white noise with zero mean and constant variance $\sigma^{2}>0$.

The model \eqref{eq:liner regression model at time t} can also be rewritten in matrix form as
\begin{align}
Y=\Phi\theta+V,
\end{align}
where
\begin{align}
Y=\left[\begin{array}{c}y(1)\\\vdots\\y(N)\end{array}\right],\ \Phi=\left[\begin{array}{c}\phi^{T}(1)\\\vdots\\ \phi^{T}(N)\end{array}\right],\ V=\left[\begin{array}{c}v(1)\\\vdots\\v(N)\end{array}\right].
\end{align}

Our goal is to estimate the unknown $\theta$ as ``good" as possible based on the historical data sets $\{y(t),\phi(t)\}_{t=1}^{N}$. Two types of mean square error (MSE)(\cite{COL2012}, \cite{PC2015}) can be used to evaluate how ``good" an estimator $\hat{\theta}\in\R^{n}$ of the true parameter $\theta_{0}\in\R^{n}$ performs, which are defined as follows,
\begin{subequations}
\begin{align}\label{eq:MSEg criterion}
{\MSE}_{g}(\hat{\theta})=&\E(\|\hat{\theta}-\theta_{0}\|_{2}^{2}),\\
\label{eq:MSEy criterion}
{\MSE}_{y}(\hat{\theta})=&\E(\|\Phi\theta_{0}+V^{*}-\Phi\hat{\theta}\|_{2}^{2}),
\end{align}
\end{subequations}
where $\E(\cdot)$ denotes the mathematical expectation, $\|\cdot\|_{2}$ denotes the Euclidean norm, and $V^{*}$ is an independent copy of $V$. The smaller MSE indicates the better performance of $\hat{\theta}$. At the same time, $\MSE_{g}$ and $\MSE_{y}$ are closely connected with each other, which is stated in \cite{MCL2018}.

Assume that $\Phi\in\R^{N\times n}$ is full column rank with $N>n$, i.e. $\rank(\Phi)=n$. One classic estimation method is the Least Squares (LS):
\begin{subequations}\label{eq:LS estimator}
\begin{align}\label{eq:LS form1}
\hat{\theta}^{\LS}=&\argmin_{\theta\in\R^{n}}\|Y-\Phi\theta\|_{2}^{2}\\
                  =&(\Phi^{T}\Phi)^{-1}\Phi^{T}Y,
\end{align}
\end{subequations}
Although the LS estimator $\hat{\theta}^{\LS}$ is unbiased, it may have large variance, which still results in large $\MSE_{g}$,
\begin{subequations}
\begin{align}
\E(\hat{\theta}^{\LS})=&\theta_{0},\\
\Var(\hat{\theta}^{\LS})=&\sigma^2\Tr[(\Phi^{T}\Phi)^{-1}],\\
\label{eq:MSEg of LS estimate}
{\MSE}_{g}(\hat{\theta}^{\LS})=&\Var(\hat{\theta}^{\LS})+\|\E(\hat{\theta}^{\LS})-\theta_{0}\|_{2}^2\nonumber\\
                            =&\sigma^2\Tr[(\Phi^{T}\Phi)^{-1}],
\end{align}
\end{subequations}
where $\Var(\cdot)$ is the mathematical variance and $\Tr(\cdot)$ denotes the trace of a square matrix.

\begin{remark}
{\it When $\Phi^{T}\Phi$ is very ill-conditioned, the performance of $\hat{\theta}^{\LS}$ will always be poor by the measure of $\MSE_{g}$. Define that eigenvalues of an $n$-by-$n$ positive definite matrix with $n\geq 2$ are $\lambda_{1}(\cdot)\geq\cdots\geq\lambda_{n}(\cdot)$ and the condition number of this matrix can be represented as $\cond(\cdot)=\lambda_{1}(\cdot)/\lambda_{n}(\cdot)$. Then we can rewrite \eqref{eq:MSEg of LS estimate} as
\begin{align}
{\MSE}_{g}(\hat{\theta}^{\LS})
=&\frac{\sigma^2}{\lambda_{1}(\Phi^{T}\Phi)}\left[1+\sum_{i=2}^{n}\frac{\lambda_1(\Phi^{T}\Phi)}{\lambda_{i}(\Phi^{T}\Phi)}\right],
\end{align}
which means that
\begin{align}
&\frac{\sigma^2}{\lambda_{1}(\Phi^{T}\Phi)}\cond(\Phi^{T}\Phi)\nonumber\\
< &{\MSE}_{g}(\hat{\theta}^{\LS}) \leq \frac{n\sigma^2}{\lambda_{1}(\Phi^{T}\Phi)}\cond(\Phi^{T}\Phi).
\end{align}
There are two factors influencing the lower bound of ${\MSE}_{g}(\hat{\theta}^{\LS})$: $\sigma^2/\lambda_{1}(\Phi^{T}\Phi)$ and $\cond(\Phi^{T}\Phi)$.
\begin{itemize}
	\item For the first factor $\sigma^2/\lambda_{1}(\Phi^{T}\Phi)$, if $\lambda_{1}(\Phi^{T}\Phi)$ is close to zero, then $\|\Phi\|_{F}$, where $\|\cdot\|_{F}$ denotes the Frobenius norm of a matrix, also becomes zero and we could not get enough valid information from outputs. Correspondingly, the estimation of $\theta$ would be very hard even if $\Phi^{T}\Phi$ is well-conditioned.
	
	\item For a fixed $\lambda_{1}(\Phi^{T}\Phi)$, as the second factor $\cond(\Phi^{T}\Phi)$ becomes larger, i.e. $\Phi^{T}\Phi$ becomes more ill-conditioned, the $\MSE_{g}(\hat{\theta}^{\LS})$ will also increase, indicating the worse performance of $\hat{\theta}^{\LS}$.
\end{itemize}
}
\end{remark}

In the following part, we use the concept of almost sure convergence. We define that the random sequence $\{\xi_{N}\}$ converges almost surely to a random variable $\xi$ if and only if $\forall \epsilon>0$, $\lim_{N\to\infty}\Prob(|\xi_{i}-\xi|>\epsilon\ \text{for}\ \text{all}\ i\geq N)=0$, which can be written as $\xi_{N}\overset{a.s.}\to \xi$ as $N\to\infty$.

\begin{remark}
{\it Moreover, $\lambda_{1}(\Phi^{T}\Phi)/\sigma^2$ can conservatively act as the signal-to-noise ratio (snr), which can be defined as the ratio of variances of the noise-free output and the noise:
	\begin{align}\label{eq:snr} \text{snr}=\frac{\frac{1}{N}\sum_{i=1}^{N}(\phi_{i}^{T}\theta_{0}-\frac{1}{N}\sum_{i=1}^{N}\phi_{i}^{T}\theta_{0})^2}{\sigma^2}.
	\end{align}
If we assume that $\{\phi_{i}\}_{i=1}^{N}$ are independent and normally distributed with zero mean and constant covariance $\Sigma\in\R^{n\times n}$, it follows that $\phi_{i}^{T}\theta_{0}\sim \mathcal{N}(0,\theta_{0}^{T}\Sigma\theta_{0})$ for $i=1,\cdots,N$. Define that the eigenvector of $\Sigma$ is $u_{i}\in\R^{n}$ corresponding to $\lambda_{i}(\Sigma)$ with $i=1,\cdots,n$. According to Corollary \ref{corollary:convergence of sample covariance matrix} in Appendix, we know that as $N\to\infty$,
\begin{align}
\text{snr} \overset{a.s.}\to& \frac{\theta_{0}^{T}\Sigma\theta_{0}}{\sigma^2}
=\frac{\sum_{i=1}^{n}\lambda_{i}(\Sigma)\theta_{0}^{T}u_{i}u_{i}^{T}\theta_{0}}{\sigma^2}
\leq \frac{\lambda_{1}(\Sigma)}{\sigma^2}\theta_{0}^{T}\theta_{0}.
\end{align}
Since $\lambda_{1}(\Phi^{T}\Phi)/N \overset{a.s.}\to \lambda_{1}(\Sigma)$ as $N\to\infty$, small $\lambda_{1}(\Phi^{T}\Phi)/\sigma^{2}$ always gives a smaller $\text{snr}$. When the $\text{snr}$ is very small, even if the condition number of $\Phi^{T}\Phi$ is equal to one, $\hat{\theta}^{\LS}$ still performs badly. We usually set $\text{snr}\geq1$ in simulation experiments.}
\end{remark}

\begin{remark}
{\it Interestingly, the ${\MSE}_{y}$ of the LS estimator
\begin{align}
{\MSE}_{y}(\hat{\theta}^{\LS})
=&\E(\|\Phi\theta_{0}+V^{*}-\Phi\hat{\theta}^{\LS}\|_{2}^{2})\nonumber\\
=&(N+n)\sigma^2
\end{align}
is irrespective of $\cond(\Phi^{T}\Phi)$.
}
\end{remark}

To handle this problem, we can introduce one regularization term in \eqref{eq:LS form1} to obtain the regularized least squares (RLS) estimator:
\begin{subequations}\label{eq:RLS estimator}
\begin{align}
\hat{\theta}^{\TR}=&\argmin_{\theta\in\R^{n}}\|Y-\Phi\theta\|_{2}^{2}+\sigma^2\theta^{T}P^{-1}\theta\\
                 =&(\Phi^{T}\Phi+\sigma^2 P^{-1})^{-1}\Phi^{T}Y\\
                 =&P\Phi^{T}Q^{-1}Y,
\end{align}
\end{subequations}
where
\begin{align}
Q=\Phi P\Phi^{T}+\sigma^2I_{N},
\end{align}
$P\in\R^{n\times n}$ is positive semidefinite and often known as the kernel matrix, and $I_{N}$ denotes the $N$-dimensional identity matrix.

\section{Kernel Design and Hyper-parameter Estimation}

For the regularization method, our main concerns are the kernel design and the hyper-parameter estimation.

\subsection{Kernel Design}

The structure of the kernel matrix $P$ should be designed based on the prior knowledge about the true system by parameterizing it with the hyper-parameter $\eta\in\R^{p}$, which can be tuned in the set $\Omega\subset\R^{p}$. Several popular positive semidefinite kernels have been invented before,
\begin{subequations}
\begin{align}\label{eq:SS kernel}
\text{SS}:&P_{i,j}(\eta)=c\left(\frac{\alpha^{i+j+\max(i,j)}}{2}-\frac{\alpha^{3\max(i,j)}}{6}\right)\nonumber\\
          &\eta=[c,\alpha]\in\Omega=\{c\geq 0,\ \alpha\in[0,1]\},\\
\label{eq:DC kernel}
\text{DC}:&P_{i,j}(\eta)=c\alpha^{(i+j)/2}\rho^{|i-j|},\nonumber\\
          &\eta=[c,\alpha,\rho]\in\Omega=\{c\geq 0,\ \alpha\in[0,1],\ |\rho|\leq 1\},\\
\label{eq:TC kernel}
\text{TC}:&P_{i,j}(\eta)=c\alpha^{\max(i,j)},\nonumber\\
          &\eta=[c,\alpha]\in\Omega=\{c\geq 0,\ \alpha\in[0,1]\},
\end{align}
\end{subequations}
where the stable spline (SS) kernel \eqref{eq:SS kernel} is firstly introduced in \cite{PD2010}, the diagonal correlated (DC) kernel \eqref{eq:DC kernel} and the tuned-correlated (TC) kernel \eqref{eq:TC kernel} (also named as the first order stable spline kernel) are introduced in \cite{COL2012}.

\subsection{Hyper-parameter Estimation}

If the structure of $P(\eta)$ has been fixed, our next step is to estimate the hyper-parameter $\eta$ using the historical data. There are many estimation approaches for the tuning of $\eta$, such as the empirical Bayes (EB) method \cite{PDCDL2014}, the Stein's unbiased estimation (SURE) method of $\MSE_{g}$ and $\MSE_{y}$ \cite{PC2015}, the generalized marginal likelihood method, the generalization cross validation (GCV) method \cite{GHW1979} and so on.

Here we mainly investigate two hyper-parameter estimation methods: $\EB$ and $\SURE_{y}$ (the SURE for $\MSE_{y}$). The $\EB$ method assumes that $\theta$ is Gaussian distributed with zero mean and covariance $P$, and $V$ is also normally distributed, i.e.,
\begin{align}
&\theta\sim \mathcal{N}({\0},P),\ V\sim \mathcal{N}({\0},\sigma^2I_{N}),\\
\Rightarrow &Y\sim \mathcal{N}({\0},\Phi P\Phi^{T}+\sigma^2I_{N}).
\end{align}
By maximizing the likelihood function of $Y$, $\EB$ can be represented as
\begin{align}\label{eq:EB hyperparameter estimator}
\text{EB}:\hat{\eta}_{\text{EB}}=&\argmin_{\eta\in\Omega}\mathcal{F}_{\EB}(\eta)\\ \label{eq:EB1}
\mathcal{F}_{\EB}=&Y^{T}Q^{-1}Y+\log\det(Q).
\end{align}
The $\SURE_{y}$ can be written as
\begin{align}
\label{eq: SURE_y hyperparameter estimator}
{\SURE}_{y}:\hat{\eta}_{\Sy}=&\argmin_{\eta\in\Omega}\mathcal{F}_{\Sy}(\eta)\\ \label{eq:SURE_y1}
\mathcal{F}_{\Sy}=&\|Y-\Phi\hat{\theta}^{\TR}(\eta)\|_{2}^{2}+2\sigma^2\Tr(\Phi P\Phi^{T} Q^{-1}).
\end{align}

Before the further analysis, we firstly make some definitions and assumptions, which are consistent with \cite{MCL2018}. Define the corresponding Oracle counterparts of $\EB$ and $\SURE_{y}$ as follows,
\begin{align}
\label{eq:EEB hyperparameter estimator}
\text{EEB}:\hat{\eta}_{\text{\EEB}}=&\argmin_{\eta\in\Omega}\mathcal{F}_{\EEB}(\eta)\\
\label{eq:EEB}
\mathcal{F}_{\EEB}=&\theta_{0}^{T}\Phi^{T}Q^{-1}\Phi\theta_{0}+\sigma^2\Tr(Q^{-1})\nonumber\\
                   &+\log\det(Q),\\
\label{eq:MSEy hyperparameter estimator}
{\MSE}_{y}:\hat{\eta}_{\MSE_{y}}=&\argmin_{\eta\in\Omega}\mathcal{F}_{\MSE_{y}}(\eta)\\
\label{eq:MSEy}
\mathcal{F}_{\MSE_{y}}=&\sigma^4\theta_{0}^{T}\Phi^{T}Q^{-2}\Phi\theta_{0}+\sigma^6\Tr(Q^{-2})\nonumber\\
                       &-2\sigma^4\Tr(Q^{-1})+2N\sigma^2.
\end{align}

\begin{assumption}\label{asp:1}
{\it The optimal hyper-parameter estimates $\hat{\eta}_{\text{EB}}$, $\hat{\eta}_{\text{Sy}}$, $\hat{\eta}_{\text{\EEB}}$ and $\hat{\eta}_{\MSE_{y}}$ are interior points of $\Omega$.}
\end{assumption}

\begin{assumption}\label{asp:almost sure convergence of PP to Sigma}
{\it $P$ is positive definite and as $N\to\infty$, $(\Phi^{T}\Phi)/N$ converges to the positive definite $\Sigma\in\R^{n\times n}$ almost surely, i.e. $(\Phi^{T}\Phi)/N \overset{a.s.}\to \Sigma\succ0$.}
\end{assumption}

Under Assumption \ref{asp:1} and \ref{asp:almost sure convergence of PP to Sigma}, we can define the limit functions of $\EB$, $\EEB$ and $\SURE_{y}$, $\MSE_{y}$ respectively,
\begin{align}\label{eq:opt hyperparameter of Wb}
\eta_{b}^{*}=&\argmin_{\eta\in\Omega}W_{b}(P,\theta_{0})\\
W_{b}(P,\theta_{0})=&\theta_{0}^{T}P^{-1}\theta_{0}+\log\det(P),\\
\label{eq:opt hyperparameter of Wy}
\eta_{y}^{*}=&\argmin_{\eta\in\Omega}W_{y}(P,\Sigma,\theta_{0})\\
W_{y}(P,\Sigma,\theta_{0})=&\sigma^4\theta_{0}^{T}P^{-T}\Sigma^{-1}P^{-1}\theta_{0}-2\sigma^4\Tr(\Sigma^{-1}P^{-1}).
\end{align}

\begin{assumption}\label{asp:2}
{\it	The sets $\eta_{b}^{*}$ and $\eta_{y}^{*}$ are made of isolated points,respectively.}
\end{assumption}

In the following assumption, we apply the concept of the boundedness in probability. Let $\xi_{N}=O_{p}(a_{N})$ denote that $\{\xi_{N}/a_{N}\}$ is bounded in probability, which means that $\forall \epsilon>0$, $\exists L>0$ such that $P(|\xi_{N}/a_{N}|>L)<\epsilon$ for any $N$.

\begin{assumption}\label{asp:boundedness in probability with deltaN}
{\it $\|(\Phi^{T}\Phi)/N-\Sigma\|_{F}=O_{p}(\delta_{N})$ and as $N\to\infty$, $\delta_{N}\to 0$.}
\end{assumption}

Under the Assumption \ref{asp:1}, \ref{asp:almost sure convergence of PP to Sigma}, \ref{asp:2} and \ref{asp:boundedness in probability with deltaN}, it has been shown in \cite{MCL2018} that:
\begin{itemize}
\item $\hat{\eta}_{\Sy}$ is asymptotically optimal, while $\hat{\eta}_{\EB}$ is not, which means that as $N\to\infty$,
\begin{align}
&\hat{\eta}_{\EB} \overset{a.s.}\to \eta_{b}^{*},\ \hat{\eta}_{\EEB} \overset{a.s.}\to \eta_{b}^{*},\\
&\hat{\eta}_{\Sy} \overset{a.s.}\to \eta_{y}^{*},\ \hat{\eta}_{\MSE_{y}} \overset{a.s.}\to \eta_{y}^{*}.
\end{align}
\item the convergence rate of $\hat{\eta}_{\Sy}$ to $\eta^{*}_{y}$ is related with the convergence rate of $(\Phi^{T}\Phi)/N$ to $\Sigma$, while that of $\hat{\eta}_{\EB}$ to $\eta^{*}_{b}$ is not, which means that
    \begin{subequations}\label{eq:convergence rate of EB and SUREy estimators}
    \begin{align}
    &\|\hat{\eta}_{\EB}-\eta_{b}^{*}\|_{2}=O_{p}(1/\sqrt{N}),\\
    &\|\hat{\eta}_{\Sy}-\eta_{y}^{*}\|_{2}=O_{p}(\mu_{N}),\\
    &\mu_{N}=\max(O_{p}(\delta_{N}),O_{p}(1/\sqrt{N})).
    \end{align}
    \end{subequations}
\end{itemize}

\begin{remark}
{\it For the $\SURE_{y}$, it will not encounter the ill-conditionedness for practical computation using the cholesky decomposition \cite{CL2013}, but the ill-conditioned issue will appear if we analyze the property of its hyper-parameter estimator. }
\end{remark}

According to the findings and simulation experiments in \cite{MCL2018}, although $\hat{\eta}_{\EB}$ is not asymptotically optimal, we can still observe better performance of $\hat{\eta}_{\EB}$ than that of $\hat{\eta}_{\Sy}$ in the sense of $\MSE_{g}$, when $\Phi^{T}\Phi$ is ill-conditioned and the sample size is small. Thus we draw attention to the influence of $\cond(\Phi^{T}\Phi)$ on the convergence rates of the cost functions and hyper-parameter estimators of $\EB$ and $\SURE_{\text{y}}$, respectively.

\section{Effects of $\cond(\Phi^{T}\Phi)$ on the Convergence Rates of Cost Functions of $\EB$ and $\SURE_{y}$}

Let
\begin{align}
\overline{\mathcal{F}_{\EB}}
=&\mathcal{F}_{\EB}+Y^{T}\Phi(\Phi^{T}\Phi)^{-1}\Phi^{T}Y/\sigma^2-Y^{T}Y/\sigma^2\nonumber\\
&-(N-n)\log\sigma^2-\log\det(\Phi^{T}\Phi)\\
=&(\hat{\theta}^{\LS})^{T}S^{-1}\hat{\theta}^{\LS}+\log\det(S),\\
\overline{\mathcal{F}_{\Sy}}=&N[\mathcal{F}_{\Sy}+Y^{T}\Phi(\Phi^{T}\Phi)^{-1}\Phi^{T}Y-Y^{T}Y-2n\sigma^2]\\
=&N\left[\sigma^4(\hat{\theta}^{\LS})^{T}S^{-T}(\Phi^{T}\Phi)^{-1}S^{-1}\hat{\theta}^{\LS}\right.\nonumber\\
&\left.-2\sigma^4\Tr((\Phi^{T}\Phi)^{-1}S^{-1})\right],\\
S=&P+\sigma^2(\Phi^{T}\Phi)^{-1}.
\end{align}
Under Assumption \ref{asp:1}, \ref{asp:almost sure convergence of PP to Sigma} and \ref{asp:2}, it has been proved in \cite{MCL2018} that as $N\to\infty$,
\begin{align}
\overline{\mathcal{F}_{\EB}} \overset{a.s.}\to W_{b},\
\overline{\mathcal{F}_{\Sy}} \overset{a.s.}\to W_{y}.
\end{align}

In fact, we can also investigate the influence of $\cond(\Phi^{T}\Phi)$ on the convergence rates of cost function by computing the upper bounds of $|\overline{\mathcal{F}_{\EB}}- W_{b}|$ and $|\overline{\mathcal{F}_{\Sy}}- W_{y}|$.

\begin{remark}\label{rmk:note for Op()}
	{\it To be clear, the first part of each upper bound in Theorem \ref{thm:convergence rate of Feb and Wb} and \ref{thm:convergence rate of Fsy and Wy} indicates its boundedness in probability. For example, as shown in \eqref{eq:upper bound of inv Phiphi} of Corollary \ref{corollary:upper bounds of building blocks} in Appendix, $\|A_{N}^{-1}\|_{F}=O_{p}(1/a_{N})$. If one term is $O_{p}(1)$, we omit this part for the convenience. }
\end{remark}

Applying Corollary \ref{corollary:upper bounds of building blocks}, the upper bounds of $|\overline{\mathcal{F}_{\EB}}- W_{b}|$ and $|\overline{\mathcal{F}_{\Sy}}- W_{y}|$ can be represented in the following Theorem \ref{thm:convergence rate of Feb and Wb} and \ref{thm:convergence rate of Fsy and Wy}.

\begin{theorem}\label{thm:convergence rate of Feb and Wb}
{\it	Under Assumption \ref{asp:1}, \ref{asp:almost sure convergence of PP to Sigma} and \ref{asp:2}, we have
\begin{align}
|\overline{\mathcal{F}_{\EB}}-W_{b}|
\leq&E_{1,b}+E_{2,b}+E_{3,b},
\end{align}
where
\begin{align}\label{eq:E_1b}
E_{1,b}=&\|\theta_{0}\|_{2}\|\Phi^{T}V\|_{2}\|(\Phi^{T}\Phi)^{-1}\|_{F}(\|S^{-1}\|_{F}+\|P^{-1}\|_{F})\\
\label{eq:E_2b}		
E_{2,b}=&\|(\Phi^{T}\Phi)^{-1}\|_{F}\left[\|S^{-1}\|_{F}\left(\|\Phi^{T}V\|_{2}^2\|(\Phi^{T}\Phi)^{-1}\|_{F}\right.\right.\nonumber\\
&+\sigma^2\|\theta_{0}\|_{2}^2\|P^{-1}\|_{F})+\sqrt{r}\sigma^2\nonumber\\
&\left.\max(\|S^{-1}\|_{F}\|P^{-1}\|_{F}\|P^{1/2}\|_{F}^2,\|P^{-1/2}\|_{F}^2)\right]\\
\label{eq:E_3b}
E_{3,b}=&\sigma^2\|\theta_{0}\|_{2}\|\Phi^{T}V\|_{2}\|(\Phi^{T}\Phi)^{-1}\|_{F}^2\|S^{-1}\|_{F}\|P^{-1}\|_{F}\\
r_{1}=&\rank(I_{n}-P^{1/2}S^{-1}P^{1/2}).
\end{align}
Upper bounds of terms \eqref{eq:E_1b}, \eqref{eq:E_2b} and \eqref{eq:E_3b} are shown as follows, respectively,
\begin{align}		E_{1,b}\leq&\frac{1}{\sqrt{N}}n\|\theta_{0}\|_{2}\frac{N}{\lambda_{1}(\Phi^{T}\Phi)}\cond(\Phi^{T}\Phi)\frac{\|\Phi^{T}V\|_{2}}{\sqrt{N}}\nonumber\\
&\left[\frac{1}{\lambda_{1}(S)}\cond(S)+\frac{1}{\lambda_{1}(P)}\cond(P)\right]\\
E_{2,b}\leq&\frac{1}{N}n^{3/2}\frac{N}{\lambda_{1}(\Phi^{T}\Phi)}\cond(\Phi^{T}\Phi)\nonumber\\ &\left[\frac{1}{\lambda_{1}(S)}\cond(S)\left(\frac{\|\Phi^{T}V\|_{2}^2}{N}\frac{N}{\lambda_{1}(\Phi^{T}\Phi)}\right.\right.\nonumber\\
&\left.\left.\cond(\Phi^{T}\Phi)+\sigma^2\|\theta_{0}\|_{2}^{2}\frac{1}{\lambda_{1}(P)}\cond(P)\right)\right.\nonumber\\		&+\left.\sqrt{r_{1}}\sigma^2\max\left(n\frac{1}{\lambda_{1}(S)}\frac{1}{\lambda_{1}(P)}\frac{1}{\lambda_{1}(S)}\cond(S)\cond(P),\right.\right.\nonumber\\
&\left.\left.\frac{1}{\lambda_{1}(P)}\cond(P)\right)\right]\\		E_{3,b}\leq&\frac{1}{N^{3/2}}n\sigma^2\|\theta_{0}\|_{2}\frac{N^2}{\lambda_{1}^2(\Phi^{T}\Phi)}\frac{1}{\lambda_{1}(S)}\frac{1}{\lambda_{1}(P)}{\cond}^2(\Phi^{T}\Phi)\nonumber\\
&\cond(S)\cond(P)\frac{\|\Phi^{T}V\|_{2}}{\sqrt{N}}.
\end{align}}
\end{theorem}

\begin{theorem}\label{thm:convergence rate of Fsy and Wy}
{\it  Under Assumption \ref{asp:1}, \ref{asp:almost sure convergence of PP to Sigma}, \ref{asp:2} and \ref{asp:boundedness in probability with deltaN}, we have
\begin{align}
|\overline{\mathcal{F}_{\Sy}}-W_{y}|
\leq&E_{1,y}+E_{2,y}+E_{3,y}+E_{4,y}+E_{5,y},
\end{align}
where
\begin{align}\label{eq:E_1y}		
E_{1,y}=&\sigma^4\|\theta_{0}\|_{2}\|\Phi^{T}V\|_{2}\|(\Phi^{T}\Phi)^{-1}\|_{F}(N\|(\Phi^{T}\Phi)^{-1}\|_{F}\|S^{-1}\|_{F}^2\nonumber\\
&+\|\Sigma^{-1}\|_{F}\|P^{-1}\|_{F}^2)\\
\label{eq:E_2y}		
E_{2,y}=&\sigma^4N\|(\Phi^{T}\Phi)^{-1}\|_{F}\left\|\frac{\Phi^{T}\Phi}{N}-\Sigma\right\|_{F}\|\Sigma^{-1}\|_{F}\|P^{-1}\|_{F}\nonumber\\
&(\|\theta_{0}\|_{2}^2\|S^{-1}\|_{F}+2\sqrt{r_{2}})\\
\label{eq:E_3y}
E_{3,y}=&\sigma^4\|(\Phi^{T}\Phi)^{-1}\|_{F}\|S^{-1}\|_{F}\nonumber\\
        &(\|\Phi^{T}V\|_{2}^{2}N\|(\Phi^{T}\Phi)^{-1}\|_{F}^2\|S^{-1}\|_{F}\nonumber\\
		&+\sigma^2\|\theta_{0}\|_{2}^2N\|(\Phi^{T}\Phi)^{-1}\|_{F}\|S^{-1}\|_{F}\|P^{-1}\|_{F}\nonumber\\
		&+\sigma^2\|\theta_{0}\|_{2}^2\|\Sigma^{-1}\|_{F}\|P^{-1}\|_{F}^2\nonumber\\
		&+2\sqrt{r_{2}}\sigma^2N\|(\Phi^{T}\Phi)^{-1}\|_{F}\|P^{-1}\|_{F})\\
\label{eq:E_4y}
E_{4,y}=&\sigma^6\|\theta_{0}\|_{2}\|\Phi^{T}V\|_{2}\|(\Phi^{T}\Phi)^{-1}\|_{F}^2\|S^{-1}\|_{F}\|P^{-1}\|_{F}\nonumber\\
		&(N\|(\Phi^{T}\Phi)^{-1}\|_{F}\|S^{-1}\|_{F}+\|\Sigma^{-1}\|_{F}\|P^{-1}\|_{F})\\
\label{eq:E_5y}
E_{5,y}=&\sigma^4\|\theta_{0}\|_{2}\|\Phi^{T}V\|_{2}N\|(\Phi^{T}\Phi)^{-1}\|_{F}^2\left\|\frac{\Phi^{T}\Phi}{N}-\Sigma\right\|_{F}\nonumber\\
        &\|\Sigma^{-1}\|_{F}\|S^{-1}\|_{F}\|P^{-1}\|_{F}
\end{align}
\begin{align}
r_{2}=&\rank(\Sigma^{-1}P^{-1}-N(\Phi^{T}\Phi)^{-1}S^{-1}).
\end{align}
Upper bounds of \eqref{eq:E_1y}, \eqref{eq:E_2y}, \eqref{eq:E_3y}, \eqref{eq:E_4y} and \eqref{eq:E_5y} are shown as follows, respectively,
\begin{align}
E_{1,y}\leq&\frac{1}{\sqrt{N}}n^2\sigma^4\|\theta_{0}\|_{2}\frac{N}{\lambda_{1}(\Phi^{T}\Phi)}\cond(\Phi^{T}\Phi) \frac{\|\Phi^{T}V\|_{2}}{\sqrt{N}}\nonumber\\		&\left(\frac{N}{\lambda_{1}(\Phi^{T}\Phi)}\frac{1}{\lambda_{1}^2(S)}\cond(\Phi^{T}\Phi){\cond}^2(S)\right.\nonumber\\
&\left.+\frac{1}{\lambda_{1}(\Sigma)}\frac{1}{\lambda_{1}^2(P)}\cond(\Sigma){\cond}^2(P)\right)\\
E_{2,y}\leq&\delta_{N}n^2\sigma^4\frac{N}{\lambda_{1}(\Phi^{T}\Phi)} \frac{\lambda_{n}({\Phi^{T}\Phi}/{N}-\Sigma)}{\delta_{N}}\frac{1}{\lambda_{1}(\Sigma)}\frac{1}{\lambda_{1}(P)}\nonumber\\		&\cond(\Phi^{T}\Phi)\cond(\frac{\Phi^{T}\Phi}{N}-\Sigma)\cond(\Sigma)\cond(P)\nonumber\\
&(\sqrt{n}\|\theta_{0}\|_{2}^{2}\frac{1}{\lambda_{1}(S)}\cond(S)+2\sqrt{r_{2}})\\
E_{3,y}\leq&\frac{1}{N}n^2\sigma^4\frac{N}{\lambda_{1}(\Phi^{T}\Phi)}\frac{1}{\lambda_{1}(S)}\cond(\Phi^{T}\Phi)\cond(S)\nonumber\\
&\left[\sqrt{n}\frac{\|\Phi^{T}V\|_{2}^2}{N}\left(\frac{N}{\lambda_{1}(\Phi^{T}\Phi)}\right)^2\right.\nonumber\\
&\left.\frac{1}{\lambda_{1}(S)}{\cond}^2(\Phi^{T}\Phi)\cond(S)\right.\nonumber\\	&+\sqrt{n}\sigma^2\|\theta_{0}\|_{2}^2\frac{N}{\lambda_{1}(\Phi^{T}\Phi)}\frac{1}{\lambda_{1}(S)}\frac{1}{\lambda_{1}(P)}\nonumber\\
&\cond(\Phi^{T}\Phi)\cond(S)\cond(P)\nonumber\\
&+\sqrt{n}\sigma^2\|\theta_{0}\|_{2}^2\frac{1}{\lambda_{1}(\Sigma)}\frac{1}{\lambda_{1}^2(P)}\cond(\Sigma){\cond}^2(P)\nonumber\\
&+\left.2\sqrt{r_{2}}\sigma^2\frac{N}{\lambda_{1}(\Phi^{T}\Phi)}\frac{1}{\lambda_{1}(P)}\cond(\Phi^{T}\Phi)\cond(P)\right]\\
E_{4,y}\leq&\frac{1}{N^{3/2}}\sigma^6n^3\|\theta_{0}\|_{2}\left(\frac{N}{\lambda_{1}(\Phi^{T}\Phi)}\right)^2\frac{1}{\lambda_{1}(S)}\frac{1}{\lambda_{1}(P)}\nonumber\\
&{\cond}^2(\Phi^{T}\Phi)\cond(S)\cond(P)\frac{\|\Phi^{T}V\|_{2}}{\sqrt{N}}\nonumber\\	&\left(\frac{N}{\lambda_{1}(\Phi^{T}\Phi)}\frac{1}{\lambda_{1}(S)}\cond(\Phi^{T}\Phi)\cond(S)\right.\nonumber\\
&+\left.\frac{1}{\lambda_{1}(\Sigma)}\frac{1}{\lambda_{1}(P)}\cond(\Sigma)\cond(P)\right)
\end{align}
\begin{align}
E_{5,y}\leq&\frac{\delta_{N}}{\sqrt{N}}\sigma^4n^3\|\theta_{0}\|_{2}\left(\frac{N}{\lambda_{1}(\Phi^{T}\Phi)}\right)^2\frac{\lambda_{n}({\Phi^{T}\Phi}/{N}-\Sigma)}{\delta_{N}}\nonumber\\
&\frac{1}{\lambda_{1}(\Sigma)}\frac{1}{\lambda_{1}(S)}\frac{1}{\lambda_{1}(P)}{\cond}^2(\Phi^{T}\Phi)\nonumber\\
&\cond(\frac{\Phi^{T}\Phi}{N}-\Sigma)\cond(\Sigma)\cond(S)\cond(P)\frac{\|\Phi^{T}V\|_{2}}{\sqrt{N}}.
\end{align}	}
\end{theorem}

\begin{remark}
{\it If $\{\phi(t)\}_{t=1}^{N}$ are assumed to be independent and normally distributed with zero mean, covariance matrix $\Sigma$ and finite fourth moment, we can derive that $\delta_{N}=1/\sqrt{N}$ using the Central Limit Theorem (CLT).}
\end{remark}

The comparison between upper bounds of $|\overline{\mathcal{F}_{\EB}}-W_{b}|$ and $|\overline{\mathcal{F}_{\Sy}}-W_{y}|$ with respect to the powers of $\cond(\Phi^{T}\Phi)$ and $\cond(P)$ is summarized in the following table.

\begin{table}[!htp]
	\centering
	\caption{Upper bounds of $|\overline{\mathcal{F}_{\EB}}-W_{b}|$ and $|\overline{\mathcal{F}_{\Sy}}-W_{y}|$}
	\label{table: comparison of Feb, Wb and Fsy,Wy}
	\begin{tabular}{cccc}
		\hline
		\multicolumn{4}{c}{$|\overline{\mathcal{F}_{\EB}}-W_{b}|$}  \\
		\hline
		\makecell[c]{\text{boundedness}\\ {\text {in}}\\ {\text{probability}}} & \text{term} &  \makecell[c]{\text{maximum power}\\ \text{of}\\ $\cond(\Phi^{T}\Phi)$}& \makecell[c]{\text{maximum power}\\ \text{of}\\ $\cond(P)$}\\
		\hline	
		 ${1}/{\sqrt{N}}$ & $E_{1,b}$ & 1&1  \\
		 ${1}/{N}$ & $E_{2,b}$ & 2 &1 \\
		 ${1}/{N^{3/2}}$ & $E_{3,b}$ & 2&1 \\
		\hline
        \hline
         \multicolumn{4}{c}{$|\overline{\mathcal{F}_{\Sy}}-W_{y}|$}\\
        \hline
        \makecell[c]{\text{boundedness}\\ \text{in}\\  \text{probability}} & \text{term} & \makecell[c]{\text{maximum power}\\ \text{of}\\ $\cond(\Phi^{T}\Phi)$}& \makecell[c]{\text{maximum power}\\ \text{of}\\ $\cond(P)$}\\
        \hline
        ${1}/{\sqrt{N}}$ & $E_{1,y}$  & 2&2\\
        ${1}/{N}$ & $E_{3,y}$  & 3&2\\
        ${1}/{N^{3/2}}$ & $E_{4,y}$  & 3&2\\
        $\delta_{N}$& $E_{2,y}$  & 2&1\\
        ${\delta_{N}}/{\sqrt{N}}$& $E_{5,y}$  & 3&1\\
        \hline
	\end{tabular}
\end{table}

\begin{remark}
	{\it Since as $N\to\infty$, $(\Phi^{T}\Phi)/N \overset{a.s.}\to \Sigma$, it can be seen that
	\begin{align}
	\cond(\Phi^{T}\Phi)=\cond\left(\frac{\Phi^{T}\Phi}{N}\right) \overset{a.s.}\to \cond(\Sigma),
	\end{align}
    which means that for any $\overline{\epsilon}>0$, $\exists\ \overline{N}>0$, then for all $N>\overline{N}$
    \begin{align}
    &|\cond(\Phi^{T}\Phi)-\cond(\Sigma)|<\overline{\epsilon}\ \text{almost}\ \text{surely}.
    \end{align}
    Then in Table \ref{table: comparison of Feb, Wb and Fsy,Wy}, for the term like $\cond(\Phi^{T}\Phi)\cond(\Sigma)$, its greatest power of $\cond(\Phi^{T}\Phi)$ is regarded as $2$.
}
\end{remark}

As shown in Table \ref{table: comparison of Feb, Wb and Fsy,Wy}, comparing the upper bounds of $|\overline{\mathcal{F}_{\EB}}-W_{b}|$ and $|\overline{\mathcal{F}_{\Sy}}-W_{y}|$, the greatest power of $\cond(\Phi^{T}\Phi)$ of $|\overline{\mathcal{F}_{\Sy}}-W_{y}|$ is always one larger than that of $|\overline{\mathcal{F}_{\EB}}-W_{b}|$, with regard to each term with the same boundedness in probability. At the same time, ill-conditioned $\Phi^{T}\Phi$ may usually result in large $\cond(P)$. Table \ref{table: comparison of Feb, Wb and Fsy,Wy} also shows that the greatest power of $\cond(P)$ in each term of $|\overline{\mathcal{F}_{\Sy}}-W_{y}|$ upper bound is one larger than that of $|\overline{\mathcal{F}_{\EB}}-W_{b}|$ upper bound, correspondingly. Thus, the large $\cond(\Phi^{T}\Phi)$ may lead to far slower convergence rate of $\overline{\mathcal{F}_{\Sy}}$ to $W_{y}$ than that of $\overline{\mathcal{F}_{\EB}}$ to $W_{b}$. It also inspires us to continue the study of the effects of large $\cond(\Phi^{T}\Phi)$ on the comparison between the convergence rate of $\hat{\eta}_{\EB}$ to $\eta_{b}^{*}$ and that of $\hat{\eta}_{\Sy}$ to $\eta_{y}^{*}$.

\section{Effects of $\cond(\Phi^{T}\Phi)$ on the Convergence Rates of Hyper-parameter Estimators of $\EB$ and $\SURE_{y}$}

 In this section, we show the asymptotic normality of $\hat{\eta}_{\EB}-\eta^{*}_{b}$ and $\hat{\eta}_{\Sy}-\eta^{*}_{y}$. Here we define that the random sequence $\{\xi_{N}\}$ converges in distribution to a random variable $\xi$ with cumulative density function (CDF) $F(\xi)$ if $\lim_{N\to\infty}|F_{N}(\xi_{N})-F(\xi)|=0$, which can be written as $\xi_{N} \overset{d}\to \xi$.

\begin{assumption}\label{asp:interior points eta_eb_star eta_sy_star}
{\it Let $\Omega$ be an open subset of the Euclidean $p$-space, which means that $\eta^{*}_{b}$ and $\eta^{*}_{y}$ are interior points of $\Omega$.}
\end{assumption}

\begin{theorem}\label{thm:asymptotic normality of eta_eb difference}
{\it Assume that the noise is Gaussian distributed, i.e. $V\sim\mathcal{N}({\0},\sigma^2I_{N})$. Under Assumption \ref{asp:1}, \ref{asp:almost sure convergence of PP to Sigma}, \ref{asp:2} and \ref{asp:interior points eta_eb_star eta_sy_star}, as $N\to\infty$, we have
\begin{align}
\sqrt{N}(\hat{\eta}_{\EB}-\eta^{*}_{b})\overset{d}\to&\mathcal{N}({\0},A_{b}(\eta^{*}_{b})^{-1}B_{b}(\eta^{*}_{b})A_{b}(\eta^{*}_{b})^{-1}),
\end{align}
where the $(k,l)$th elements of $A_{b}(\eta^{*}_{b})$ and $B_{b}(\eta^{*}_{b})$ can be represented as follows, respectively,
\begin{align}\label{eq:element of A eta_b_star}
A_{b}(\eta^{*}_{b})_{k,l}=&\left\{\theta_{0}^{T}\frac{\partial^2 P^{-1}}{\partial\eta_{k}\partial\eta_{l}}\theta_{0}
                     +{\Tr}\left(\frac{\partial P^{-1}}{\partial \eta_{l}}\frac{\partial P}{\partial \eta_{k}}\right)\right.\nonumber\\
                     &+\left.\left.\Tr\left(P^{-1}\frac{\partial^2 P}{\partial\eta_{k}\partial\eta_{l}} \right)\right\}\right|_{\eta^{*}_{b}}\\
\label{eq:element of B eta_b_star}
B_{b}(\eta^{*}_{b})_{k,l}=&4\sigma^2\left.\left\{\theta_{0}^{T}\frac{\partial P^{-1}}{\partial \eta_{k}}\Sigma^{-1}\frac{\partial P^{-1}}{\partial \eta_{l}}\theta_{0}\right\}\right|_{\eta^{*}_{b}}.
\end{align}
}
\end{theorem}

\begin{theorem}\label{thm:asymptotic normality of eta_sy difference}
{\it In addition to Assumption \ref{asp:1}, \ref{asp:almost sure convergence of PP to Sigma}, \ref{asp:2}, \ref{asp:boundedness in probability with deltaN}, \ref{asp:interior points eta_eb_star eta_sy_star} and the Gaussian noise assumption, we further suppose that $\delta_{N}=o(1/\sqrt{N})$, which means that $\delta_{N}$ is an infinitesimal of higher order than $1/\sqrt{N}$ as $N\to\infty$. (In particular, if $\delta_{N}$ is represented as $N^{k}$, $k$ should be smaller than $-1/2$.) As $N\to\infty$, we have
\begin{align}
\sqrt{N}(\hat{\eta}_{\Sy}-\eta^{*}_{y})\overset{d}\to&\mathcal{N}({\0},C_{y}(\eta^{*}_{y})^{-1}D_{y}(\eta^{*}_{y})C_{y}(\eta^{*}_{y})^{-1}),
\end{align}
where the $(k,l)$th elements of $C_{y}(\eta^{*}_{y})$ and $D_{y}(\eta^{*}_{y})$ can be represented as follows, respectively,
\begin{align}
\label{eq:element of C eta_y_star}
C_{y}(\eta^{*}_{y})_{k,l}
=&2\sigma^4\left\{\theta_{0}^{T}\frac{\partial P^{-1}}{\partial\eta_{l}}{\Sigma}^{-1}\frac{\partial P^{-1}}{\partial\eta_{k}}\theta_{0}
+\theta_{0}^{T}P^{-1}{\Sigma}^{-1}\frac{\partial^2 P^{-1}}{\partial\eta_{k}\partial\eta_{l}}\theta_{0}\right.\nonumber\\
&-\left.\left.\Tr\left({\Sigma}^{-1}\frac{\partial^2 P^{-1}}{\partial\eta_{k}\partial\eta_{l}}\right)\right\}\right|_{\eta^{*}_{y}}\\
\label{eq:element of D eta_y_star}
D_{y}(\eta^{*}_{y})_{k,l}
=&4\sigma^{10}\left\{\theta_{0}^{T}\left[P^{-1}{\Sigma}^{-1}\frac{\partial P^{-1}}{\partial\eta_{k}} + \frac{\partial P^{-1}}{\partial\eta_{k}}{\Sigma}^{-1}P^{-1} \right]{\Sigma}^{-1}\right.\nonumber\\
&\left.\left.\left[P^{-1}{\Sigma}^{-1}\frac{\partial P^{-1}}{\partial\eta_{l}} + \frac{\partial P^{-1}}{\partial\eta_{l}}{\Sigma}^{-1}P^{-1} \right]\theta_{0}\right\}\right|_{\eta^{*}_{y}}.
\end{align}
}
\end{theorem}

Under the assumption that $\delta_{N}=o(1/\sqrt{N})$ as $N\to\infty$, we know from \eqref{eq:convergence rate of EB and SUREy estimators} that the convergence rates of $\EB$ and $\SURE_{y}$ hyper-parameter estimators have the same boundedness in probability $1/\sqrt{N}$ but with different scaling coefficients, which depend on their respective asymptotic covariance matrices. Since it may be hard to shed light on the comparison between asymptotic covariance matrices of $\hat{\eta}_{\EB}-\eta_{b}^{*}$ and $\hat{\eta}_{\Sy}-\eta_{y}^{*}$ straightforwardly, we make an attempt with the ridge regression case.

\begin{corollary}\label{corollary:ridge regression case for asymptotic normality}
{\it Suppose that $P=\eta I_{n}$ and $n\geq 2$, where $\eta\in\R$. Then under assumptions of Theorem \ref{thm:asymptotic normality of eta_eb difference} and \ref{thm:asymptotic normality of eta_sy difference}, as $N\to\infty$, we have
\begin{align}\label{eq:asymptotic normality of eta_eb for the ridge regression case}
\sqrt{N}(\hat{\eta}_{\EB}-\eta^{*}_{b})\overset{d}\to&\mathcal{N}(0,\frac{4\sigma^2}{n^2}\theta_{0}^{T}\Sigma^{-1}\theta_{0})\\
\label{eq:asymptotic normality of eta_sy for the ridge regression case}
\sqrt{N}(\hat{\eta}_{\Sy}-\eta^{*}_{y})\overset{d}\to&\mathcal{N}(0,\frac{4\sigma^2}{\Tr^2(\Sigma^{-1})}\theta_{0}^{T}\Sigma^{-3}\theta_{0}).
\end{align}
As $\lambda_{n}(\Sigma)\to 0$ and other eigenvalues $\lambda_{i}(\Sigma)$ with $i=1,\cdots,n-1$ are fixed, which leads to $\cond(\Sigma)\to\infty$, the ratio of two limiting variances in \eqref{eq:asymptotic normality of eta_eb for the ridge regression case} and \eqref{eq:asymptotic normality of eta_sy for the ridge regression case} tends to be $1/n^2$, i.e.
\begin{align}
\frac{\theta_{0}^{T}\Sigma^{-1}\theta_{0}/{n^2}}{\theta_{0}^{T}\Sigma^{-3}\theta_{0}/{\Tr^2(\Sigma^{-1})}}\to\frac{1}{n^2}.
\end{align}
}
\end{corollary}

It implies that even for the ridge regression case, when $\cond(\Sigma)\to\infty$ and $n\geq 2$, the asymptotic variance of $\hat{\eta}_{\Sy}-\eta_{y}^{*}$ still tends to be $n^2$ times larger than that of $\hat{\eta}_{\EB}-\eta_{b}^{*}$. Then in this case, the convergence rate of $\hat{\eta}_{\Sy}$ to $\eta_{y}^{*}$ is slower than that of $\hat{\eta}_{\EB}$ to $\eta_{b}^{*}$.

\section{Numerical Simulation}

To generate data sets, we construct $\{\phi(t)\}_{t=1}^{N}$ as independent and Gaussian distributed vectors with zero mean and fixed covariance $\Sigma$. Then it satisfies $(\Phi^{T}\Phi)/N \overset{a.s.}\to \Sigma$ as $N\to\infty$, which can be proved by Corollary \ref{corollary:convergence of sample covariance matrix}. It is worth to note that under our simulation settings, $\delta_{N}=1/\sqrt{N}$, which is worse than the assumption in Theorem \ref{thm:asymptotic normality of eta_sy difference}.

In our simulation experiments, we consider the ridge regression case and set $n=50$, $\cond(\Sigma)=1\times10^{5}$ and $snr=5$. The number of Monte Carlo simulations is selected as $1\times10^{3}$. Define that $\theta_{0}\triangleq\left[\begin{array}{ccc}g_{1}&\cdots&g_{n}\end{array}\right]^{T}$ and $V^{*}\triangleq\left[\begin{array}{ccc}v(1)^{*}&\cdots&v(N)^{*}\end{array}\right]^{T}$. The performance of $\hat{\theta}^{\TR}$ in \eqref{eq:RLS estimator} can be evaluated by relative criteria \cite{Ljung1995} as follows,
\begin{align}
\text{Fit}_{g}(\hat{\theta}^{\TR},\theta_{0})=&100\times\left(1-\frac{\|\hat{\theta}^{\TR}-\theta_{0}\|_{2}}{\|\theta_{0}-\overline{\theta_{0}}\|_{2}} \right)\\
\text{Fit}_{y}(\hat{\theta}^{\TR},\theta_{0})=&100\times\left(1-\frac{\|\Phi\hat{\theta}^{\TR}-\Phi\theta_{0}-V^{*}\|_{2}}{\|\Phi\theta_{0}+V^{*}-\overline{Y^{*}}\|_{2}} \right),
\end{align}
where
\begin{align}
\overline{\theta_{0}}=\frac{1}{n}\sum_{i=1}^{n}g_{i},\ \overline{Y^{*}}=\frac{1}{N}\sum_{i=1}^{N}[\phi(i)^{T}\theta_{0}+v(i)^{*}].
\end{align}
In fact, $\text{Fit}_{g}$ evaluates the performance of $\hat{\theta}^{\TR}$ in the sense of $\MSE_{g}$ and $\text{Fit}_{y}$ measures in the sense of $\MSE_{y}$. The convergences of $\Phi^{T}\Phi/N$ to $\Sigma$, $\hat{\eta}_{\EB}$ to $\eta_{b}^{*}$, $\hat{\eta}_{\Sy}$ to $\eta^{*}_{y}$, $\overline{\mathcal{F}_{\EB}}$ to $W_{b}$ and $\overline{\mathcal{F}_{\Sy}}$ to $W_{y}$ are also evaluated by the measure of fit similarly.

\begin{figure}[thpb]
\centering
\includegraphics[width=0.5\textwidth]{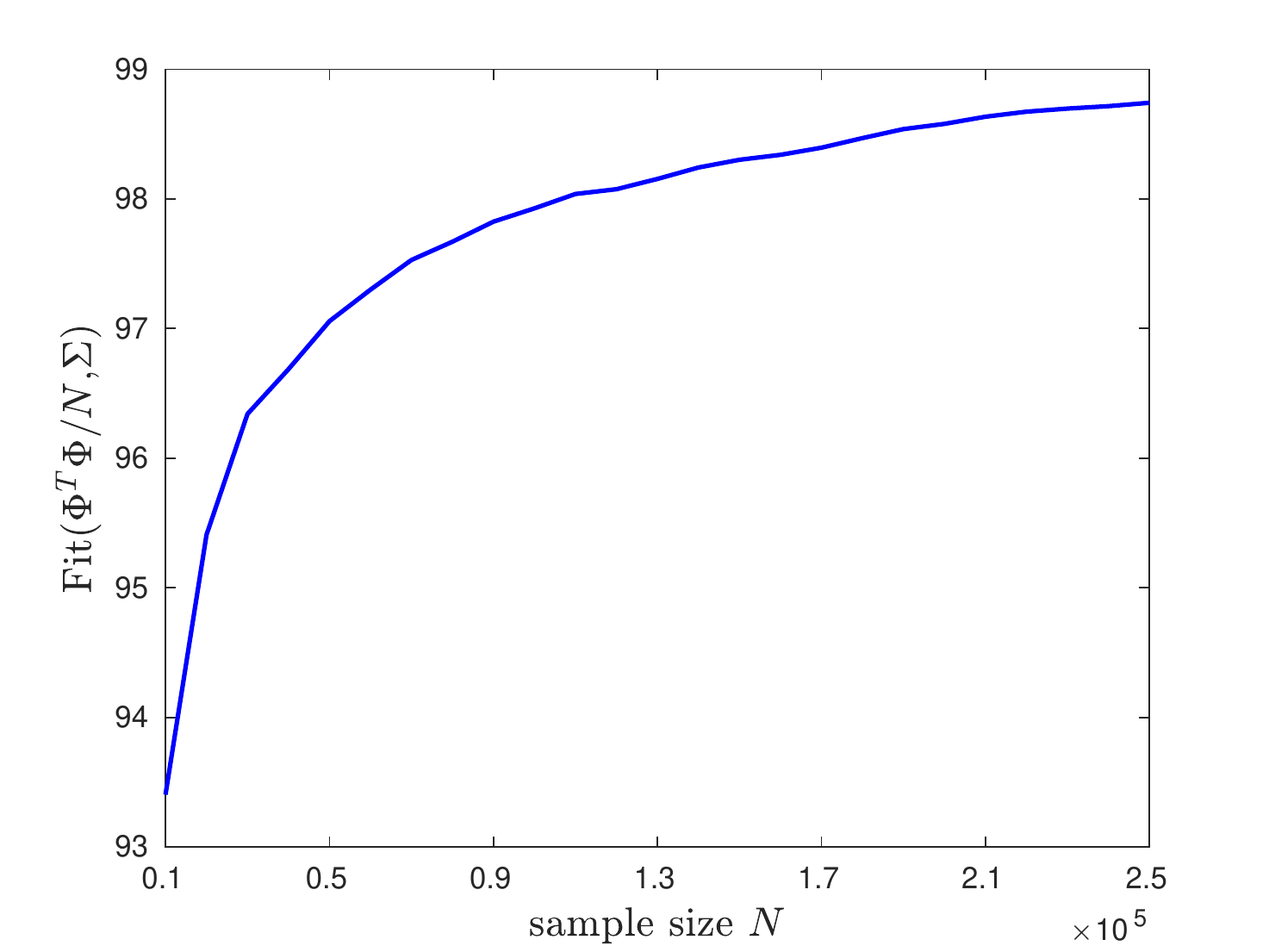}
\caption{Convergence of $(\Phi^{T}\Phi)/N$ to $\Sigma$}
\label{fig:convergence of pp}
\end{figure}

\begin{figure}[thpb]
\centering
\includegraphics[width=0.5\textwidth]{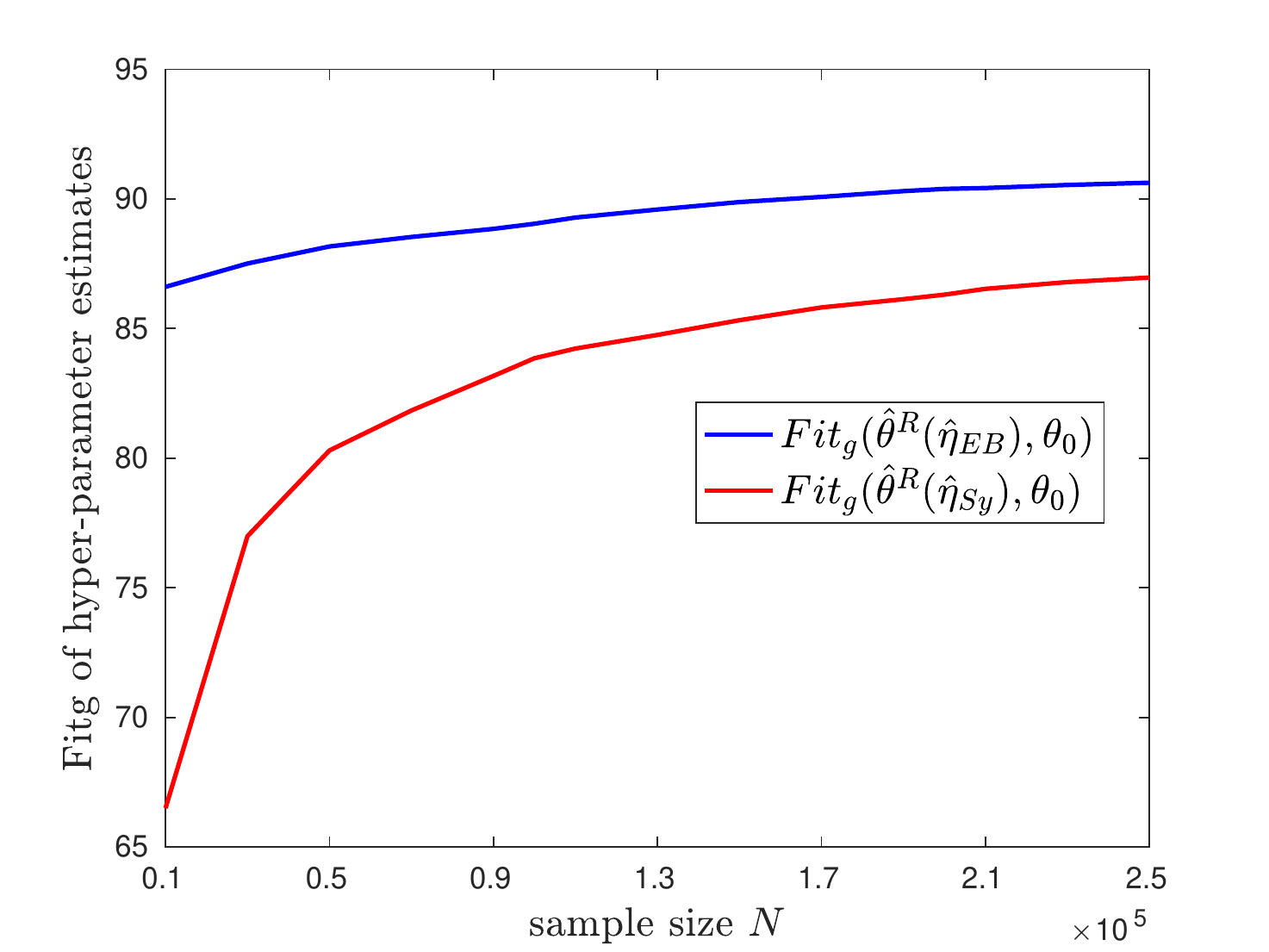}
\caption{Average $\text{Fit}_{g}$ of $\hat{\theta}^{\TR}(\hat{\eta}_{\EB})$ and $\hat{\theta}^{\TR}(\hat{\eta}_{\Sy})$}
\label{fig:fits of theta}
\end{figure}

\begin{figure}[thpb]
\centering
\includegraphics[width=0.5\textwidth]{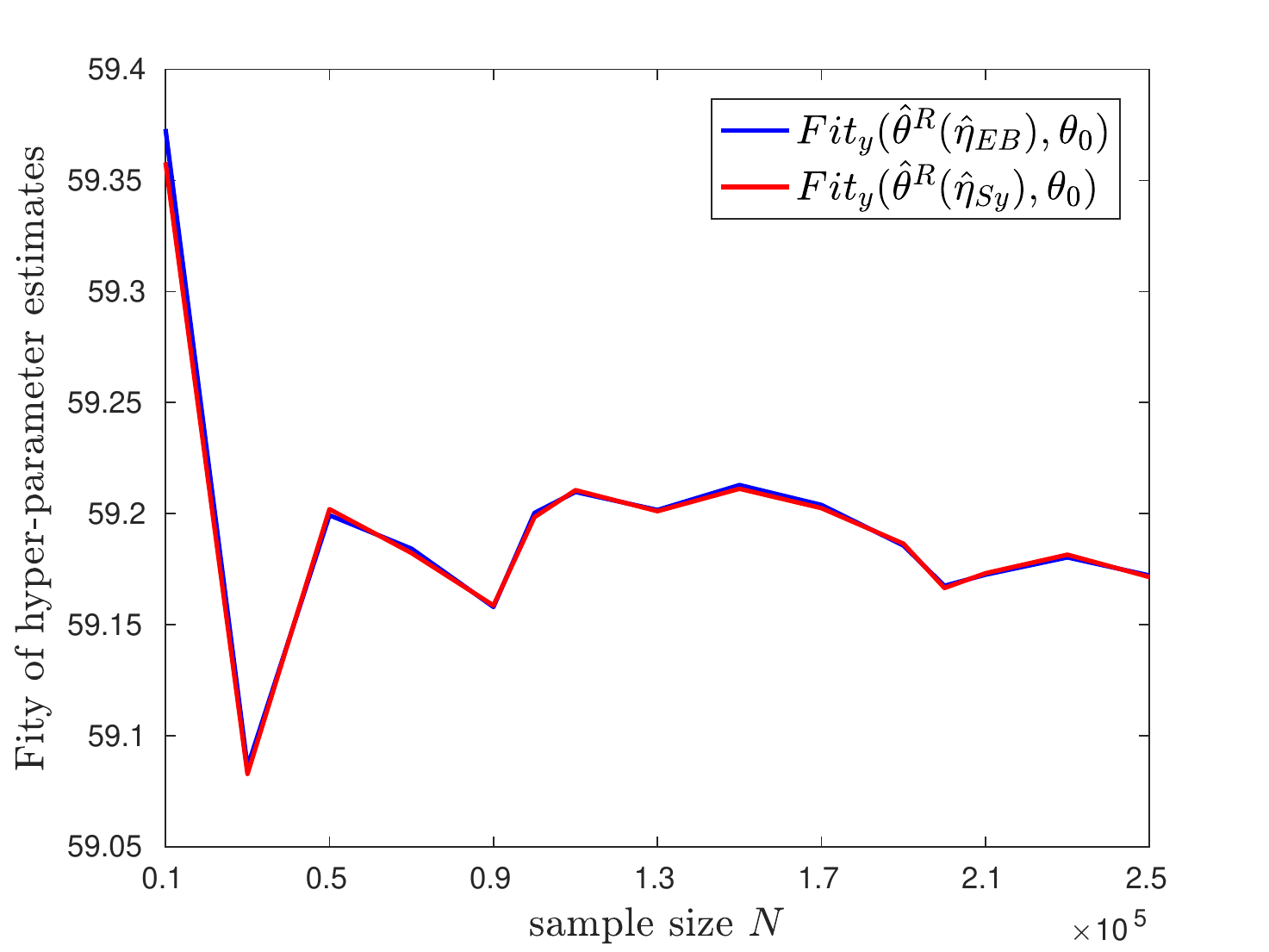}
\caption{Average $\text{Fit}_{y}$ of $\hat{\theta}^{\TR}(\hat{\eta}_{\EB})$ and $\hat{\theta}^{\TR}(\hat{\eta}_{\Sy})$}
\label{fig:fits of Y}
\end{figure}

\begin{figure}[thpb]
\centering
\includegraphics[width=0.5\textwidth]{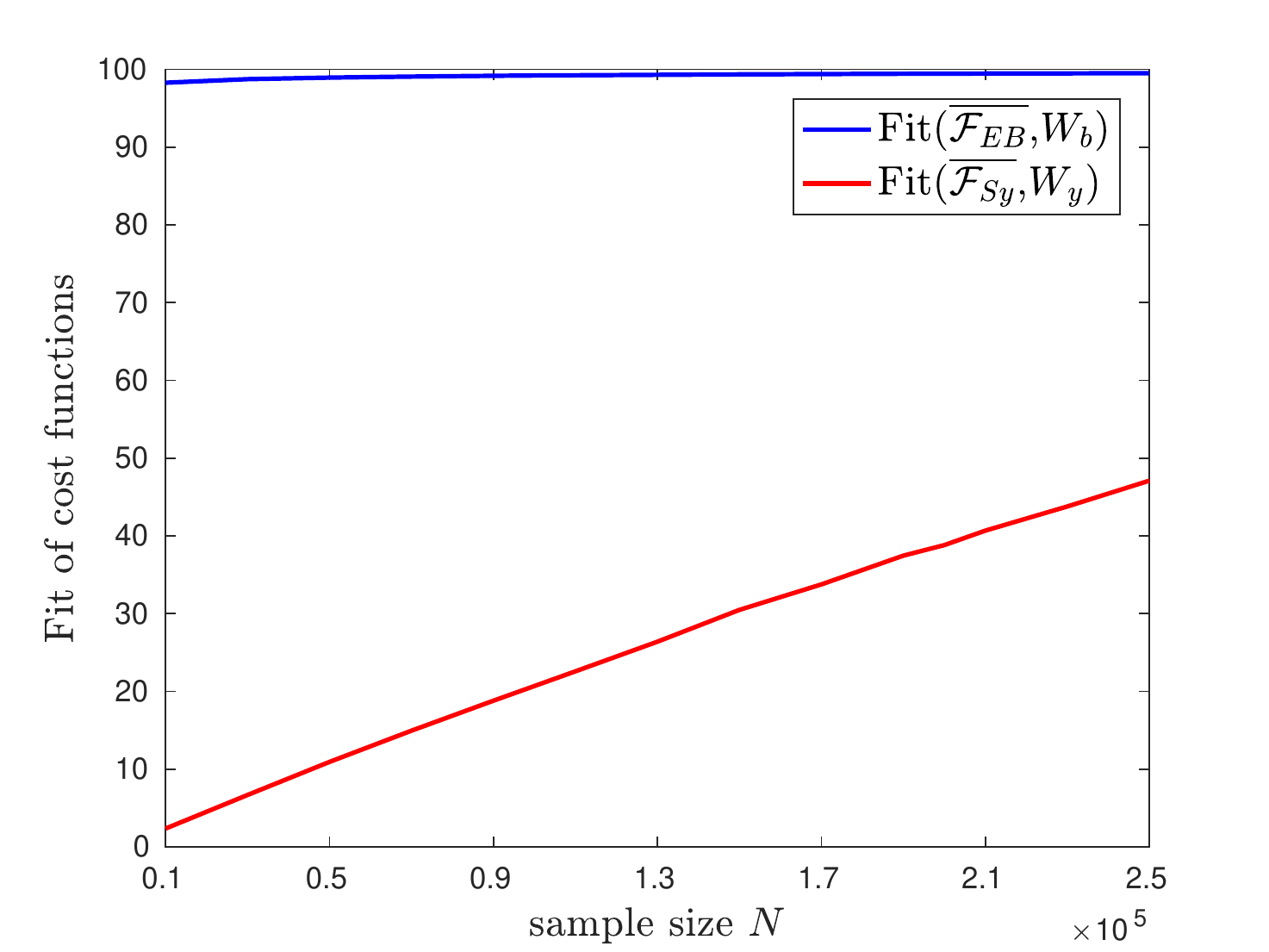}
\caption{Average fits of cost functions}
\label{fig:fits of F}
\end{figure}

\begin{figure}[thpb]
\centering
\includegraphics[width=0.5\textwidth]{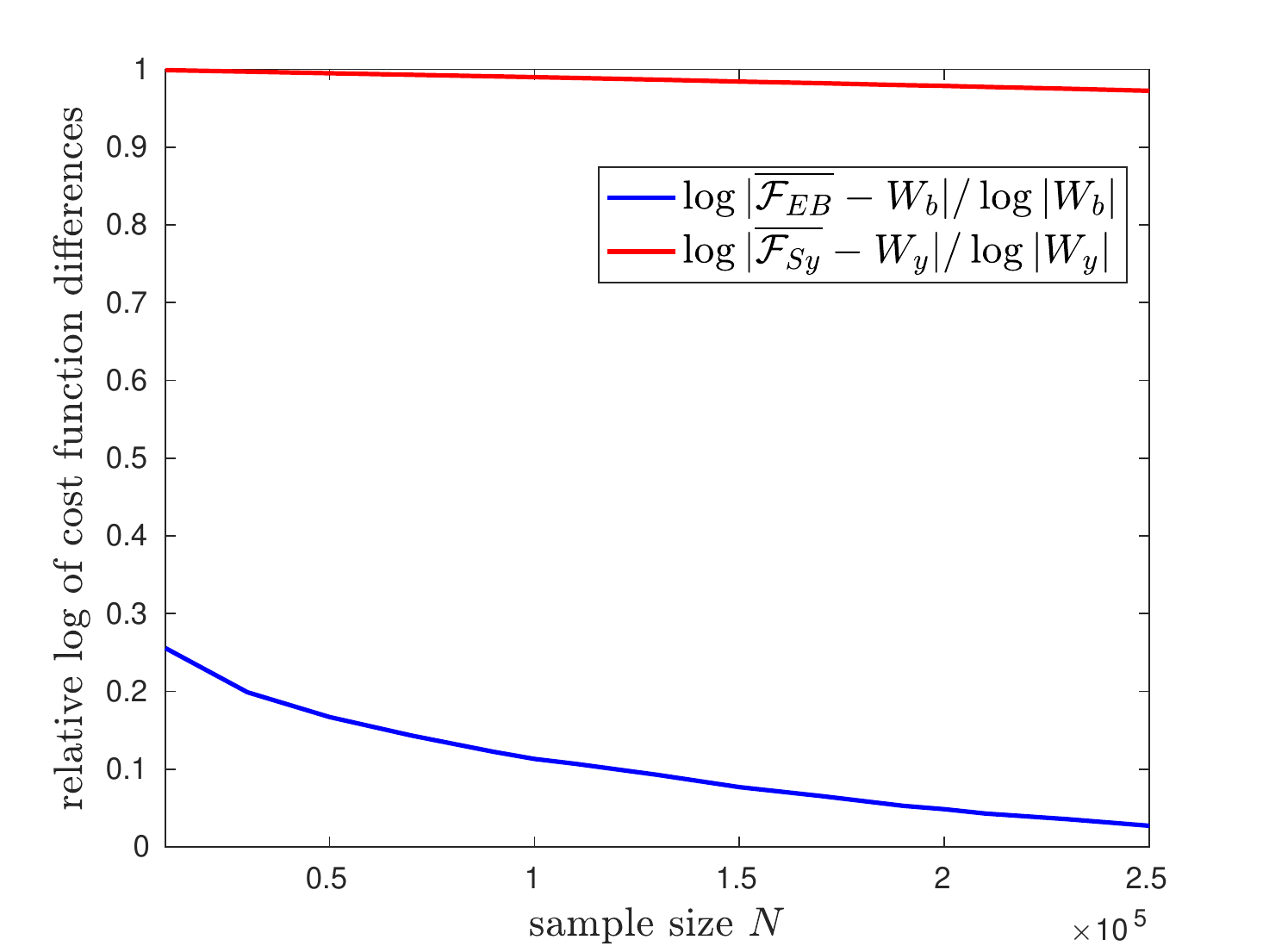}
\caption{Logarithm of absolute cost function differences}
\label{fig:logs of F}
\end{figure}

\begin{figure}[thpb]
\centering
\includegraphics[width=0.5\textwidth]{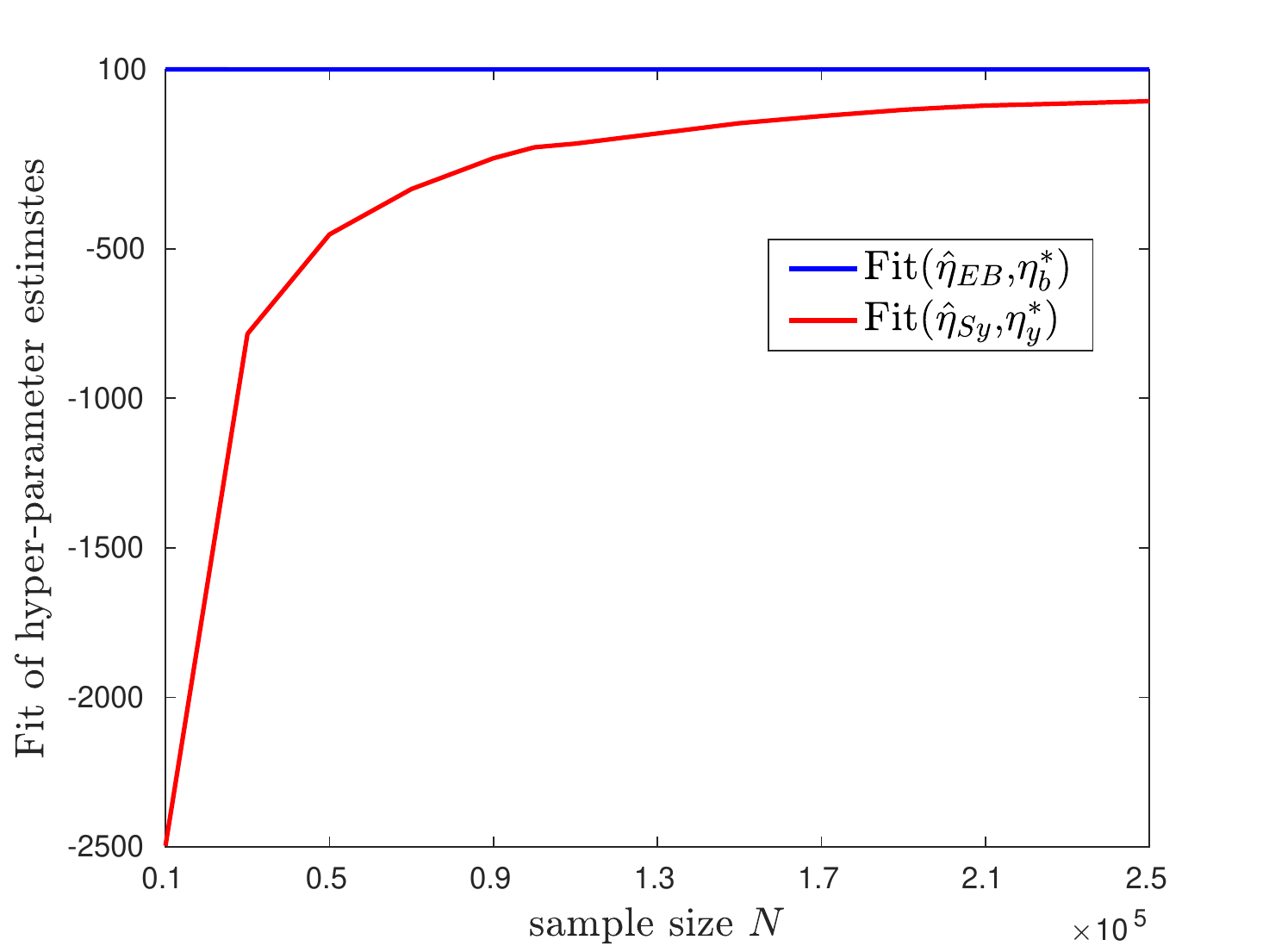}
\caption{Average fits of hyper-parameter estimates}
\label{fig:fits of eta}
\end{figure}

\begin{figure}[thpb]
\centering
\includegraphics[width=0.5\textwidth]{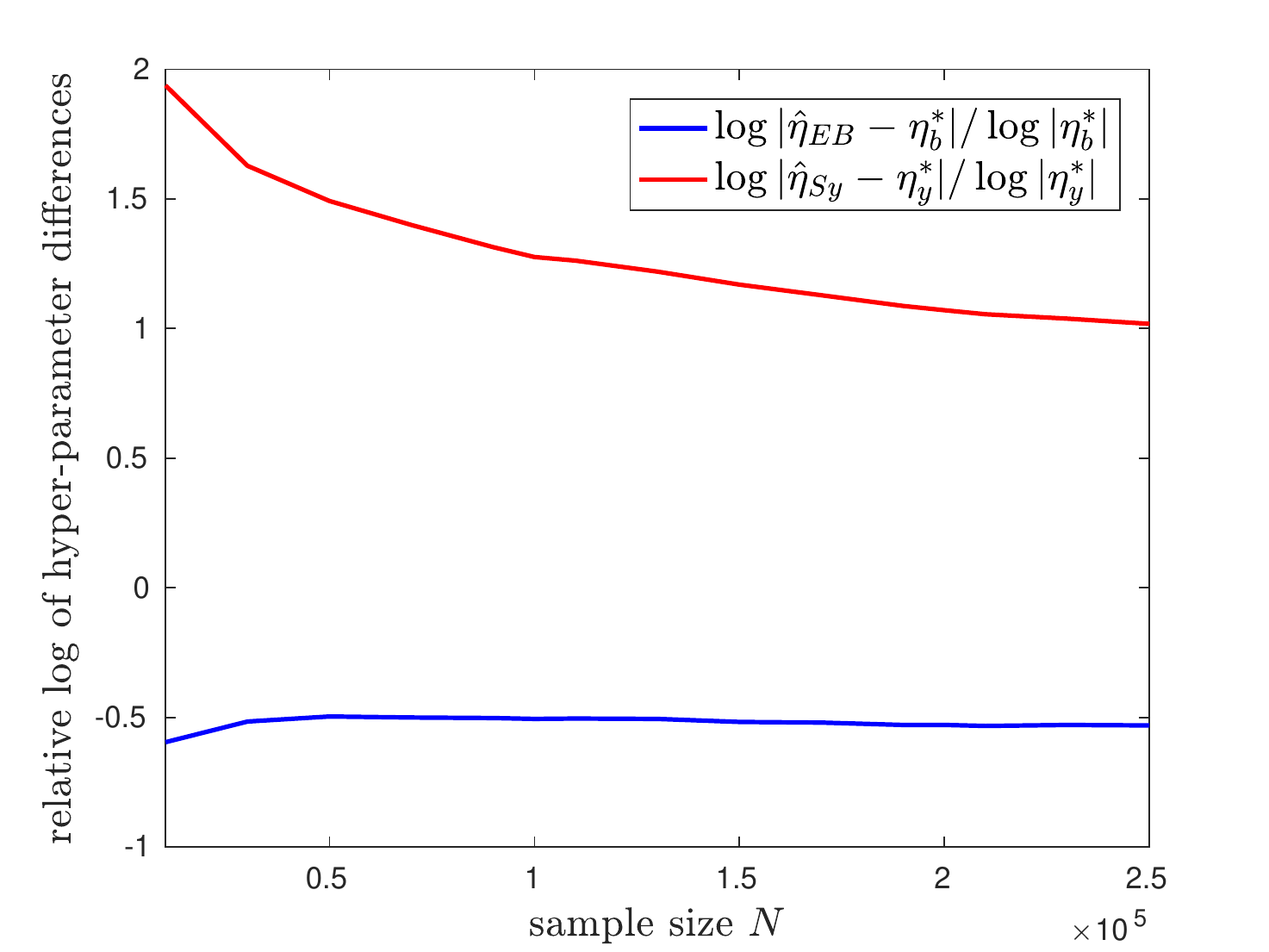}
\caption{Logarithm of absolute hyper-parameter estimate differences}
\label{fig:logs of eta}
\end{figure}

Firstly, from Fig \ref{fig:convergence of pp}, we can see that as the sample size $N$ becomes larger, $\text{Fit}((\Phi^{T}\Phi)/N,\Sigma)$ increases gradually up to $99$, which means that $(\Phi^{T}\Phi)/N$ tends to converge to $\Sigma$. It verifies the consistency of our simulation settings and Assumption \ref{asp:almost sure convergence of PP to Sigma}.

Secondly, Fig \ref{fig:fits of theta} shows that the performance of $\hat{\theta}^{\TR}(\hat{\eta}_{\EB})$ is better than that of $\hat{\theta}^{\TR}(\hat{\eta}_{\Sy})$ in the sense of $\MSE_{g}$. But in Fig \ref{fig:fits of Y}, the overall $\text{Fit}_{y}$s of $\hat{\theta}^{\TR}(\hat{\eta}_{\EB})$ and $\hat{\theta}^{\TR}(\hat{\eta}_{\Sy})$ are almost identical for all sample size, indicating that $\cond(\Phi^{T}\Phi)$ may exert few influence on $\MSE_{y}$ of the RLS estimator.

Thirdly, according to Fig \ref{fig:fits of F}, it can be observed that when $\cond(\Phi^{T}\Phi)$ is very close to $10^5$, $\overline{\mathcal{F}_{\EB}}$ converges to $W_{b}$ much faster than $\overline{\mathcal{F}_{\Sy}}$ to $W_{y}$. At the same time, the intercept of the vertical axis in Fig \ref{fig:logs of F} is around $0.75$, which indicates that $|\overline{\mathcal{F}_{\EB}}-W_{b}|$ converges to zero faster than that of $|\overline{\mathcal{F}_{\Sy}}-W_{y}|$.

Lastly, according to Fig \ref{fig:fits of eta}, it can be observed that when $\cond(\Phi^{T}\Phi)$ is very close to $10^5$, the convergence rate of $\hat{\eta}_{\EB}$ to $\eta_{b}^{*}$ is much faster than that of $\hat{\eta}_{\Sy}$ to $\eta_{y}^{*}$. The vertical intercept of Fig \ref{fig:logs of eta} is about $2.5$, which also reflects faster convergence rate of $\|\hat{\eta}_{\EB}-\eta_{b}^{*}\|_{2}$ than $\|\hat{\eta}_{\Sy}-\eta_{y}^{*}\|_{2}$ to zero.

\section{CONCLUSIONS}

In this paper, we focus on the comparison between two hyper-parameter estimation methods: $\EB$ and $\SURE_{\text{y}}$ with an emphasis on the influence of $\cond(\Phi^{T}\Phi)$, where $\cond(\cdot)$ denotes the condition number and $\Phi$ is the regression matrix. Our major results are about the contrast between convergence rates of two pairs, $\overline{\mathcal{F}_{\EB}}$ to $W_{b}$ and $\overline{\mathcal{F}_{\Sy}}$ to $W_{y}$, and $\hat{\eta}_{\EB}$ to $\eta_{b}^{*}$ and $\hat{\eta}_{\Sy}$ to $\eta_{y}^{*}$, respectively.

\begin{enumerate}
\item Comparing terms with the same boundedness in probability, the greatest power of $\cond(\Phi^{T}\Phi)$ of the upper bound of $|\overline{\mathcal{F}_{\Sy}}-W_{y}|$ is always one larger than that of the upper bound of $|\overline{\mathcal{F}_{\EB}}-W_{b}|$. It indicates that the ill-conditioned $\Phi^{T}\Phi$ may result in far slower convergence rate of $\overline{\mathcal{F}_{\Sy}}$ to $W_{y}$ than that of $\overline{\mathcal{F}_{\EB}}$ to $W_{b}$.

\item As the sample size $N\to\infty$, under the assumption of $\delta_{N}=o(1/\sqrt{N})$ and Gaussian distributed noise, we prove the asymptotic normality of $\hat{\eta}_{\EB}-\eta_{b}^{*}$ and $\hat{\eta}_{\Sy}-\eta_{y}^{*}$. In this case, $\|\hat{\eta}_{\EB}-\eta_{b}^{*}\|_{2}$ and $\|\hat{\eta}_{\Sy}-\eta_{y}^{*}\|_{2}$ are both bounded in probability with $1/\sqrt{N}$ but may have different scaling coefficients, which are connected with their asymptotic covariance matrices. For the ridge regression case, we derive that, as $\cond(\Phi^{T}\Phi)$ tends to infinity, the asymptotic variance of $\hat{\eta}_{\Sy}-\eta_{y}^{*}$ tends to be $n^2$ times larger than that of $\hat{\eta}_{\EB}-\eta_{b}^{*}$, where $n$ is the number of parameters to be estimated.
\end{enumerate}

These findings provide more insights into the influences of $\cond(\Phi^{T}\Phi)$ upon different performances of $\EB$ and $\SURE_{y}$. Our next step is to relax the restriction $\delta_{N}=o(1/\sqrt{N})$ as $N\to\infty$ to further explore the convergence rates of $\EB$ and $\SURE_{y}$ estimators with different boundedness in probability.


\section*{APPENDIX A}

Proofs of Theorem \ref{thm:convergence rate of Feb and Wb}, \ref{thm:convergence rate of Fsy and Wy}, \ref{thm:asymptotic normality of eta_eb difference} and \ref{thm:asymptotic normality of eta_sy difference}, and Corollary \ref{corollary:ridge regression case for asymptotic normality} are shown in Appendix A.

\subsection{Proof of Theorem \ref{thm:convergence rate of Feb and Wb}}

The difference of $\overline{\mathcal{F}_{\EB}}$ and $W_{b}$ can be represented as
\begin{align}\label{eq:difference of Feb and Wb}
\overline{\mathcal{F}_{\EB}}-W_{b}=D_{1,b}+D_{2,b},
\end{align}
where
\begin{align}\label{eq:term1 difference of Feb and Wb}	D_{1,b}=&(\hat{\theta}^{\LS})^{T}S^{-1}\hat{\theta}^{\LS}-\theta_{0}^{T}P^{-1}\theta_{0}\nonumber\\	=&(\hat{\theta}^{\LS}-\theta_{0})^{T}S^{-1}\hat{\theta}^{\LS}+\theta_{0}^{T}(S^{-1}-P^{-1})\hat{\theta}^{\LS}\nonumber\\	&+\theta_{0}^{T}P^{-1}(\hat{\theta}^{\LS}-\theta_{0})\\
\label{eq:term2 difference of Feb and Wb}
D_{2,b}=&\log\det(S)-\log\det(P)\nonumber\\
=&\log\det(SP^{-1})\nonumber\\
=&\log\det(P^{-1/2}SP^{-1/2}).
\end{align}
	
\begin{itemize}
\item \underline{Computation of $\|\hat{\theta}^{\LS}-\theta_{0}\|_{2}$}
		
Using Lemma \ref{lemma:matrix norm inequality}, it can be known that
\begin{align} \|\hat{\theta}^{\LS}-\theta_{0}\|_{2}=&\|(\Phi^{T}\Phi)^{-1}\Phi^{T}V\|_{2}\nonumber\\
\leq&\|(\Phi^{T}\Phi)^{-1}\|_{F}\|\Phi^{T}V\|_{2}.
\end{align}
The upper bound of $\|(\Phi^{T}\Phi)^{-1}\|_{F}=O_{p}(1/N)$ can be derived by applying Corollary \ref{corollary:upper bounds of building blocks}. We can also derive that $\|\Phi^{T}V\|_{2}=O_{p}(\sqrt{N})$ using CLT. It can be seen that
\begin{align}		
\|\hat{\theta}^{\LS}-\theta_{0}\|_{2}
\leq&\frac{1}{\sqrt{N}}\frac{\sqrt{n}N}{\lambda_{1}(\Phi^{T}\Phi)}\cond(\Phi^{T}\Phi)\frac{\|\Phi^{T}V\|_{2}}{\sqrt{N}}.
\end{align}
		
\item \underline{Computation of $\|\hat{\theta}^{\LS}\|_{2}$}
		
Since
\begin{align}		\|\hat{\theta}^{\LS}\|_{2}=&\|\hat{\theta}^{\LS}-\theta_{0}+\theta_{0}\|_{2}\nonumber\\
\leq&\|\hat{\theta}^{\LS}-\theta_{0}\|_{2}+\|\theta_{0}\|_{2},
\end{align}
it leads to that
\begin{align}	\|\hat{\theta}^{\LS}\|_{2}\leq\|\theta_{0}\|_{2}+\frac{1}{\sqrt{N}}\frac{N}{\lambda_{1}(\Phi^{T}\Phi)}\cond(\Phi^{T}\Phi)\frac{\|\Phi^{T}V\|_{2}}{\sqrt{N}}.
\end{align}
		
\item \underline{Computation of $\|S^{-1}-P^{-1}\|_{F}$}
		
It can be derived that
\begin{align}
\|S^{-1}-P^{-1}\|_{F}
=&\|S^{-1}(P-S)P^{-1}\|_{F}\nonumber\\
=&\|\sigma^2S^{-1}(\Phi^{T}\Phi)^{-1}P^{-1}\|_{F}\nonumber\\
=&\|\sigma^2S^{-1}(\Phi^{T}\Phi)^{-1}P^{-1}\|_{F}\nonumber\\		\leq&\sigma^2\|S^{-1}\|_{F}\|(\Phi^{T}\Phi)^{-1}\|_{F}\|P^{-1}\|_{F}\nonumber\\		\leq&\frac{1}{N}\frac{\sigma^2(\sqrt{n})^3N}{\lambda_{1}(\Phi^{T}\Phi)}\frac{1}{\lambda_{1}(S)}\frac{1}{\lambda_{1}(P)}\nonumber\\
&\cond(\Phi^{T}\Phi)\cond(S)\cond(P).
\end{align}
\end{itemize}
	
For $D_{1,b}$ in \eqref{eq:term1 difference of Feb and Wb}, we can apply the inequalities above directly
\begin{align}
|D_{1,b}|
\leq&\|\theta_{0}\|_{2}\|\Phi^{T}V\|_{2}\|(\Phi^{T}\Phi)^{-1}\|_{F}(\|S^{-1}\|_{F}+\|P^{-1}\|_{F})\nonumber\\
&+\|(\Phi^{T}\Phi)^{-1}\|_{F}\|S^{-1}\|_{F}(\|\Phi^{T}V\|_{2}^2\|(\Phi^{T}\Phi)^{-1}\|_{F}\nonumber\\
&+\sigma^2\|\theta_{0}\|_{2}^2\|P^{-1}\|_{F})\nonumber\\
\label{eq:lb of D1b}
&+\sigma^2\|\theta_{0}\|_{2}\|\Phi^{T}V\|_{2}\|(\Phi^{T}\Phi)^{-1}\|_{F}^2\|S^{-1}\|_{F}\|P^{-1}\|_{F}.
\end{align}

Since $P^{-1/2}SP^{-1/2}$ is positive definite, we can see that
\begin{align}
&\Tr(I_{n}-P^{1/2}S^{-1}P^{1/2})\nonumber\\
\leq&\log\det(P^{-1/2}SP^{-1/2})\leq\Tr(P^{-1/2}SP^{-1/2}-I_{n}),
\end{align}
which yields
\begin{align}
|D_{2,b}|=&|\log\det(P^{-1/2}SP^{-1/2})|\nonumber\\
\leq&\max\{|\Tr(I_{n}-P^{1/2}S^{-1}P^{1/2})|,|\Tr(P^{-1/2}SP^{-1/2}-I_{n})|\}.
\end{align}
Define
\begin{align}
\rank(I_{n}-P^{1/2}S^{-1}P^{1/2})=r_{1}.
\end{align}
It follows that
\begin{align}
&|D_{2,b}|\nonumber\\
\leq&\max\{\sqrt{r_{1}}\|I_{n}-P^{1/2}S^{-1}P^{1/2}\|_{F},\sqrt{r_{1}}\|P^{-1/2}SP^{-1/2}-I_{n}\|_{F}\}\nonumber\\
=&\max\{\sqrt{r_{1}}\|P^{1/2}(P^{-1}-S^{-1})P^{1/2}\|_{F},\sqrt{r_{1}}\|P^{-1/2}(S-P)P^{-1/2}\|_{F}\}\nonumber\\
\label{eq:lb of D2b}
\leq&\max\{\sqrt{r_{1}}\sigma^2\|(\Phi^{T}\Phi)^{-1}\|_{F}\|S^{-1}\|_{F}\|P^{-1}\|_{F}\|P^{1/2}\|_{F}^{2},\nonumber\\
&\sqrt{r_{1}}\sigma^2\|(\Phi^{T}\Phi)^{-1}\|_{F}\|P^{-1/2}\|_{F}^2\}.
\end{align}
	
Combining \eqref{eq:lb of D1b} with \eqref{eq:lb of D2b}, it can be known that
\begin{align}
|\overline{\mathcal{F}_{\EB}}-W_{b}|
\leq&|D_{1,b}|+|D_{2,b}|\nonumber\\
\leq&E_{1,b}+E_{2,b}+E_{3,b},
\end{align}
where
\begin{align}
E_{1,b}=&\|\theta_{0}\|_{2}\|\Phi^{T}V\|_{2}\|(\Phi^{T}\Phi)^{-1}\|_{F}(\|S^{-1}\|_{F}+\|P^{-1}\|_{F})\\
E_{2,b}=&\|(\Phi^{T}\Phi)^{-1}\|_{F}\nonumber\\
&\left[\|S^{-1}\|_{F}(\|\Phi^{T}V\|_{2}^2\|(\Phi^{T}\Phi)^{-1}\|_{F}+\sigma^2\|\theta_{0}\|_{2}^2\|P^{-1}\|_{F})\right.\nonumber\\
&+\left.\sqrt{r_{1}}\sigma^2\max(\|S^{-1}\|_{F}\|P^{-1}\|_{F}\|P^{1/2}\|_{F}^2,\|P^{-1/2}\|_{F}^2)\right]\\
E_{3,b}=&\sigma^2\|\theta_{0}\|_{2}\|\Phi^{T}V\|_{2}\|(\Phi^{T}\Phi)^{-1}\|_{F}^2\|S^{-1}\|_{F}\|P^{-1}\|_{F}.
\end{align}
The upper bounds of $E_{1,b}$, $E_{2,b}$ and $E_{3,b}$ can be derived using Corollary \ref{corollary:upper bounds of building blocks}.

\subsection{Proof of Theorem \ref{thm:convergence rate of Fsy and Wy}}

The difference of $\overline{\mathcal{F}_{\Sy}}$ and $W_{y}$ can be represented as
\begin{align}\label{eq:difference of Fsy and Wy}
\overline{\mathcal{F}_{\Sy}}-W_{y}=D_{1,y}+\Tr(D_{2,y}),
\end{align}
where
\begin{align}\label{eq:term1 difference of Fsy and Wy}
D_{1,y}	=&\sigma^4(\hat{\theta}^{\LS})^{T}S^{-T}N(\Phi^{T}\Phi)^{-1}S^{-1}\hat{\theta}^{\LS}\nonumber\\
&-\sigma^4\theta_{0}^{T}P^{-T}\Sigma^{-1}P^{-1}\theta_{0}\nonumber\\	=&\sigma^4(\hat{\theta}^{\LS}-\theta_{0})^{T}S^{-T}N(\Phi^{T}\Phi)^{-1}S^{-1}\hat{\theta}^{\LS}\nonumber\\	&+\sigma^4\theta_{0}^{T}(S^{-1}-P^{-1})^{T}N(\Phi^{T}\Phi)^{-1}S^{-1}\hat{\theta}^{\LS}\nonumber\\	&+\sigma^4\theta_{0}^{T}P^{-T}(N(\Phi^{T}\Phi)^{-1}-\Sigma^{-1})S^{-1}\hat{\theta}^{\LS}\nonumber\\
&+\sigma^4\theta_{0}^{T}P^{-T}\Sigma^{-1}(S^{-1}-P^{-1})\hat{\theta}^{\LS}\nonumber\\
&+\sigma^4\theta_{0}^{T}P^{-T}\Sigma^{-1}P^{-1}(\hat{\theta}^{\LS}-\theta_{0})\\
\label{eq:term2 difference of Fsy and Wy}
D_{2,y}
=&2\sigma^4(\Sigma^{-1}P^{-1}-N(\Phi^{T}\Phi)^{-1}S^{-1})\nonumber\\
=&2\sigma^4(\Sigma^{-1}-N(\Phi^{T}\Phi)^{-1})P^{-1}\nonumber\\
&+2\sigma^4N(\Phi^{T}\Phi)^{-1}(P^{-1}-S^{-1}).
\end{align}
	
\begin{itemize}
\item \underline{Computation of $\|\Sigma^{-1}\|_{F}$}
		
\begin{align}
\|\Sigma^{-1}\|_{F}\leq\frac{\sqrt{n}}{\lambda_{1}(\Sigma)}\cond(\Sigma).
\end{align}
		
\item \underline{Computation of $\|N(\Phi^{T}\Phi)^{-1}-\Sigma^{-1}\|_{F}$}
		
Since $\|(\Phi^{T}\Phi)/N-\Sigma\|_{F}=O_{p}(\delta_{N})$, we can know that
\begin{align}		\left\|\frac{\Phi^{T}\Phi}{N}-\Sigma\right\|_{F}\leq\delta_{N}\frac{\sqrt{n}\lambda_{n}({\Phi^{T}\Phi}/{N}-\Sigma)}{\delta_{N}}\cond(\frac{\Phi^{T}\Phi}{N}-\Sigma).
\end{align}
Furthermore,
\begin{align}
&\|N(\Phi^{T}\Phi)^{-1}-\Sigma^{-1}\|_{F}\nonumber\\		=&\left\|N(\Phi^{T}\Phi)^{-1}\left(\frac{\Phi^{T}\Phi}{N}-\Sigma\right)\Sigma^{-1}\right\|_{F}\nonumber\\		\leq&\|N(\Phi^{T}\Phi)^{-1}\|_{F}\left\|\frac{\Phi^{T}\Phi}{N}-\Sigma\right\|_{F}\|\Sigma^{-1}\|_{F}\nonumber\\		\leq&\delta_{N}\frac{n^{3/2}N}{\lambda_{1}(\Phi^{T}\Phi)}\frac{\lambda_{n}({\Phi^{T}\Phi}/{N}-\Sigma)}{\delta_{N}}\frac{1}{\lambda_{1}(\Sigma)}\nonumber\\
&\cond(\Phi^{T}\Phi)\cond(\frac{\Phi^{T}\Phi}{N}-\Sigma)\cond(\Sigma).
\end{align}
\end{itemize}

Suppose that
\begin{align}
r_{2}=\rank(D_{2,y})=\rank(\Sigma^{-1}P^{-1}-N(\Phi^{T}\Phi)^{-1}S^{-1}).
\end{align}

Thus, the absolute difference term can be rewritten as
\begin{align}
|\overline{\mathcal{F}_{\Sy}}-W_{y}|
=&|D_{1,y}+\Tr(D_{2,y})|\nonumber\\
\leq&|D_{1,y}|+\sqrt{r_{2}}\|D_{2,y}\|_{F}.
\end{align}
Its upper bound can similarly be obtained with the computation of building blocks above.

\subsection{Proof of Theorem \ref{thm:asymptotic normality of eta_eb difference}}

We firstly derive the asymptotic normality of $\hat{\eta}_{\EB}-\eta_{b}^{*}$ using Lemma \ref{lemma:asymptotic normality of a consistent root} with $M_{N}=N\overline{\mathcal{F}_{\EB}}$ and $\eta^{*}=\eta^{*}_{b}$.

\begin{itemize}
\item \underline{assumptions \ref{asp1 in lemma:Euclidean space}, \ref{asp2 in lemma:measurability of F} and \ref{asp4 in lemma:continuity of 2nd derivative} in Lemma \ref{lemma:asymptotic normality of a consistent root}}

The first task is to show that $N\overline{\mathcal{F}_{\EB}}(Y,\eta)$ is a measurable function of $Y$ for all $\eta\in\Omega$. Recall that
\begin{align}
\overline{\mathcal{F}_{\EB}}
=&(\hat{\theta}^{\LS})^{T}S^{-1}\hat{\theta}^{\LS}+\log\det(S)\nonumber\\
=&Y^{T}Q^{-1}\Phi(\Phi^{T}\Phi)^{-1}\Phi^{T}Y\nonumber\\
&+\log\det[P+\sigma^2(\Phi^{T}\Phi)^{-1}].
\end{align}
We can observe that $N\overline{\mathcal{F}_{\EB}}(Y,\eta)$ is a continuous function of $Y$, which leads to that $N\overline{\mathcal{F}_{\EB}}(Y,\eta)$ is also a measurable function of $Y$, $\forall \eta\in\Omega$.

Then we show that $\frac{\partial N\overline{\mathcal{F}_{\EB}}}{\partial\eta}$ exists and is continuous in an open neighbourhood of $\eta^{*}_{b}$, and $\frac{\partial^2 N\overline{\mathcal{F}_{\EB}}}{\partial\eta\partial\eta^{T}}$ exists and is continuous in an open and convex neighbourhood of $\eta_{b}^{*}$.

For common kernel structures, like SS \eqref{eq:SS kernel}, DC \eqref{eq:DC kernel} and TC \eqref{eq:TC kernel}, $P(\eta)$ is a continuous, differentiable and second-order differentiable function of $\eta$ in $\Omega$. Meanwhile, the first-order derivative and the second-order derivative of $P(\eta)$ with respect to $\eta$ are both continuous for all $\eta\in\Omega$. Then under Assumption \ref{asp:interior points eta_eb_star eta_sy_star}, it can be derived that there exists an open neighbourhood of $\eta^{*}_{b}$ such that $\frac{\partial N\overline{\mathcal{F}_{\EB}}}{\partial\eta}$ exists and is continuous. There also exists an open and convex neighbourhood of $\eta_{b}^{*}$, in which $\frac{\partial^2 N\overline{\mathcal{F}_{\EB}}}{\partial\eta\partial\eta^{T}}$ exists and is continuous.

\item \underline{assumption \ref{asp3 in lemma:convergence of F} in Lemma \ref{lemma:asymptotic normality of a consistent root}}

In this part, we prove that $\overline{\mathcal{F}_{\EB}}(\eta)$ converges to $W_{b}(\eta)$ in probability and uniformly in one neighbourhood of $\eta_{b}^{*}$.

As shown in \eqref{eq:difference of Feb and Wb}, \eqref{eq:term1 difference of Feb and Wb} and \eqref{eq:term2 difference of Feb and Wb}, $\overline{\mathcal{F}_{\EB}}-W_{b}$ can be divided into two parts: $D_{1,b}$ and $\Tr(D_{2,b})$. Recall that
\begin{align}
D_{1,b} =&(\hat{\theta}^{\LS}-\theta_{0})^{T}S^{-1}\hat{\theta}^{\LS}+\theta_{0}^{T}(S^{-1}-P^{-1})\hat{\theta}^{\LS}\nonumber\\	&+\theta_{0}^{T}P^{-1}(\hat{\theta}^{\LS}-\theta_{0})\\
D_{2,b}=&\log\det(S)-\log\det(P).
\end{align}
Under Assumption \ref{asp:interior points eta_eb_star eta_sy_star}, there exists $\overline{\Omega}_{1}\subset\Omega$ containing $\eta_{b}^{*}$ such that $0<d_{1}\leq\|P(\eta)\|_{F}\leq d_{2}<\infty$ and $\|S^{-1}\|_{F}\leq\|P^{-1}\|_{F}\leq 1/d_{1}$ for all $\eta\in\overline{\Omega}_{1}$. Noting that as $N\to\infty$, $(\Phi^{T}\Phi)/N \overset{a.s.}\to\Sigma$, it gives that $\hat{\theta}^{\LS} \overset{a.s.}\to \theta_{0}$, $S^{-1} \overset{a.s.}\to P^{-1}$ as $N\to\infty$, which can be derived as follows,
\begin{align}
S^{-1}-P^{-1}
=&-S^{-1}(S-P)P^{-1}\nonumber\\
=&-\sigma^2S^{-1}(\Phi^{T}\Phi)^{-1}P^{-1}\overset{a.s.}\to 0\\
\hat{\theta}^{\LS}-\theta_{0}
=&N(\Phi^{T}\Phi)^{-1}\Phi^{T}V/N \overset{a.s.}\to 0.
\end{align}
It also follows that $\|\hat{\theta}^{\LS}-\theta_{0}\|_{2}=O_{p}(1/\sqrt{N})$, $\|\hat{\theta}^{\LS}\|_{2}=O_{p}(1)$ and $\|S^{-1}-P^{-1}\|_{F}=O_{p}(1/N)$.
Then we can see that each term of $D_{1,b}$ and $D_{2,b}$ converge to zero almost surely and uniformly $\forall \eta\in\overline{\Omega}_{1}$. In addition, the almost sure convergence can lead to the convergence in probability. Thus, $\overline{\mathcal{F}_{\EB}}(\eta)$ converges to $W_{b}(\eta)$ in probability and uniformly for all $\eta$ in $\overline{\Omega}_{1}$.

We can also show that $W_{b}(\eta)$ attains a strict local minimum at $\eta_{b}^{*}$. If $\eta_{b}^{*}$ in \eqref{eq:opt hyperparameter of Wb} is an interior point in $\Omega$, it should satisfy the first order optimality condition of $W_{b}(\eta)$, i.e.
\begin{align}
\left.\frac{\partial W_{b}(\eta)}{\partial \eta}\right|_{\eta_{b}^{*}}=0.
\end{align}
Combining with Assumption \ref{asp:2}, $\eta_{b}^{*}$ can strictly and locally minimize $W_{b}(\eta)$.

\item \underline{assumption \ref{asp5 in lemma:convergence of 2nd derivative} in Lemma \ref{lemma:asymptotic normality of a consistent root}}

Our aim in this part is to prove that $\left.\frac{\partial^2 \overline{\mathcal{F}_{\EB}}}{\partial \eta\partial\eta^{T}}\right|_{\overline{\eta}_{N}}$ converges to $A_{b}(\eta_{b}^{*})$ in probability for any sequence $\overline{\eta}_{N}$ such that ${\lim}_{N\to\infty}\overline{\eta}_{N}=\eta^{*}_{b}$ in probability, where
\begin{align}
A_{b}(\eta^{*}_{b})\triangleq{\plim}_{N\to\infty}\E\left.\left[\frac{\partial^2 \overline{\mathcal{F}_{\EB}}}{\partial \eta\partial\eta^{T}}\right]\right|_{\eta_{b}^{*}}.
\end{align}
Here $\plim$ denotes the limiting in probability.

Our proof consists of two steps.

The first step is to show that $\left.\frac{\partial^2 \overline{\mathcal{F}_{\EB}}}{\partial \eta\partial\eta^{T}}\right|_{\overline{\eta}_{N}}$ converges to $\left.\frac{\partial^2 W_{b}}{\partial \eta\partial\eta^{T}}\right|_{\eta^{*}_{b}}$ in probability for any sequence $\overline{\eta}_{N}$ such that ${\plim}_{N\to\infty}\overline{\eta}_{N}=\eta^{*}_{b}$.

The $(k,l)$th elements of the Hessian matrices of $\overline{\mathcal{F}_{\EB}}$ and $W_{b}$ are shown as follows, respectively,
\begin{align}
\frac{\partial^2 \overline{\mathcal{F}_{\EB}}}{\partial\eta_{k}\partial\eta_{l}}
    =&(\hat{\theta}^{\LS})^{T}\frac{\partial^2 S^{-1}}{\partial\eta_{k}\partial\eta_{l}}(\hat{\theta}^{\LS})
    +\Tr\left(\frac{\partial S^{-1}}{\partial\eta_{l}}\frac{\partial P}{\partial\eta_{k}}\right)\nonumber\\
    &+\Tr\left(S^{-1}\frac{\partial^2 P}{\partial\eta_{k}\partial\eta_{l}}\right)\\
\label{eq:2nd derivatives of Wb}
\frac{\partial^2 W_{b}}{\partial\eta_{k}\partial\eta_{l}}
    =&\theta_{0}^{T}\frac{\partial^2 P^{-1}}{\partial\eta_{k}\partial\eta_{l}}\theta_{0}+\Tr\left(\frac{\partial P^{-1}}{\partial\eta_{l}}\frac{\partial P}{\partial\eta_{k}}\right)\nonumber\\
    &+\Tr\left(P^{-1}\frac{\partial^2 P}{\partial\eta_{k}\partial\eta_{l}}\right).
\end{align}
Then the $(k,l)$th element of the difference between $\frac{\partial^2 \overline{\mathcal{F}_{\EB}}}{\partial \eta\partial\eta^{T}}$ and $\frac{\partial^2 W_{b}}{\partial \eta\partial\eta^{T}}$ can be represented as
\begin{align}
\frac{\partial^2 \overline{\mathcal{F}_{\EB}}}{\partial\eta_{k}\partial\eta_{l}}-\frac{\partial^2 W_{b}}{\partial\eta_{k}\partial\eta_{l}}
=&\Psi_{1,b}+\Tr(\Psi_{2,b}),
\end{align}
where
\begin{align}
\Psi_{1,b}
=&(\hat{\theta}^{\LS}-\theta_{0})^{T}\frac{\partial^2 S^{-1}}{\partial\eta_{k}\partial\eta_{l}}\hat{\theta}^{\LS}\nonumber\\
&+\theta_{0}^{T}\left(\frac{\partial^2 S^{-1}}{\partial\eta_{k}\partial\eta_{l}}-\frac{\partial^2 P^{-1}}{\partial\eta_{k}\partial\eta_{l}}\right)\hat{\theta}^{\LS}\nonumber\\
&+\theta_{0}^{T}\frac{\partial^2 P^{-1}}{\partial\eta_{k}\partial\eta_{l}}(\hat{\theta}^{\LS}-\theta_{0})\\
\Psi_{2,b}=&\left(\frac{\partial S^{-1}}{\partial\eta_{l}}-\frac{\partial P^{-1}}{\partial\eta_{l}}\right)\frac{\partial P}{\partial\eta_{k}}+(S^{-1}-P^{-1})\frac{\partial P}{\partial\eta_{k}}.
\end{align}
Under Assumption \ref{asp:interior points eta_eb_star eta_sy_star}, there exists a neighborhood $\overline{\Omega}_{2}\subset\Omega$ of $\eta_{b}^{*}$ such that for any $k=1,\cdots,p$ and $l=1,\cdots,p$, $\frac{\partial^2 P}{\partial\eta_{k}\partial\eta_{l}}$, $\frac{\partial P}{\partial\eta_{k}}$ and $P$ are all bounded. Since $\|S^{-1}\|_{F}\leq\|P^{-1}\|_{F}$ and
\begin{align}
\frac{\partial^2 P^{-1}}{\partial\eta_{k}\partial\eta_{l}}
=&P^{-1}\frac{\partial P}{\partial \eta_{l}}P^{-1}\frac{\partial P}{\partial \eta_{k}}P^{-1}
-P^{-1}\frac{\partial^2 P^{-1}}{\partial\eta_{k}\partial\eta_{l}}P^{-1}\nonumber\\
&+P^{-1}\frac{\partial P}{\partial \eta_{k}}P^{-1}\frac{\partial P}{\partial\eta_{l}}P^{-1}\\
\frac{\partial^2 S^{-1}}{\partial\eta_{k}\partial\eta_{l}}
=&S^{-1}\frac{\partial P}{\partial \eta_{l}}S^{-1}\frac{\partial P}{\partial \eta_{k}}S^{-1}
-S^{-1}\frac{\partial^2 P^{-1}}{\partial\eta_{k}\partial\eta_{l}}S^{-1}\nonumber\\
&+S^{-1}\frac{\partial P}{\partial \eta_{k}}S^{-1}\frac{\partial P}{\partial\eta_{l}}S^{-1},
\end{align}
it follows that $\frac{\partial^2 P^{-1}}{\partial\eta_{k}\partial\eta_{l}}$ and $\frac{\partial^2 S^{-1}}{\partial\eta_{k}\partial\eta_{l}}$ are both bounded $\forall\eta\in\overline{\Omega}_{2}$ with $k=1,\cdots,p$ and $l=1,\cdots,p$. As $N\to\infty$, since $(\Phi^{T}\Phi)/N \overset{a.s.}\to\Sigma$, we have $\hat{\theta}^{\LS}\overset{a.s.}\to\theta_{0}$, $S^{-1}\overset{a.s.}\to P^{-1}$, $\frac{\partial S^{-1}}{\partial\eta_{k}}\overset{a.s.}\to \frac{\partial P^{-1}}{\partial\eta_{k}}$ and $\frac{\partial^2 S^{-1}}{\partial\eta_{k}\partial\eta_{l}}\overset{a.s.}\to\frac{\partial^2 P^{-1}}{\partial\eta_{k}\partial\eta_{l}}$. The last two can be proved as follows,
\begin{align}
&\frac{\partial S^{-1}}{\partial \eta_{k}}-\frac{\partial P^{-1}}{\partial \eta_{k}}\nonumber\\
=&-S^{-1}\frac{\partial P}{\partial \eta_{k}}S^{-1}+P^{-1}\frac{\partial P}{\partial \eta_{k}}P^{-1}\nonumber\\
=&(P^{-1}-S^{-1})\frac{\partial P}{\partial \eta_{k}}P^{-1}\nonumber\\
&+S^{-1}\frac{\partial P}{\partial \eta_{k}}(P^{-1}-S^{-1})\overset{a.s.}\to 0
\end{align}
\begin{align}
&\frac{\partial^2 S^{-1}}{\partial\eta_{k}\partial\eta_{l}}-\frac{\partial^2 P^{-1}}{\partial\eta_{k}\partial\eta_{l}}\nonumber\\
=&(S^{-1}-P^{-1})\frac{\partial P}{\partial \eta_{l}}S^{-1}\frac{\partial P}{\partial \eta_{k}}S^{-1}\nonumber\\
&+P^{-1}\frac{\partial P}{\partial \eta_{l}}(S^{-1}-P^{-1})\frac{\partial P}{\partial \eta_{k}}S^{-1}\nonumber\\
&+P^{-1}\frac{\partial P}{\partial\eta_{l}}P^{-1}\frac{\partial P}{\partial \eta_{k}}(S^{-1}-P^{-1})\nonumber\\
&+(P^{-1}-S^{-1})\frac{\partial^2 P^{-1}}{\partial\eta_{k}\partial\eta_{l}}P^{-1}\nonumber\\
&+S^{-1}\frac{\partial^2 P^{-1}}{\partial\eta_{k}\partial\eta_{l}}(P^{-1}-S^{-1})\nonumber\\
&+(S^{-1}-P^{-1})\frac{\partial P}{\partial \eta_{k}}S^{-1}\frac{\partial P}{\partial \eta_{l}}S^{-1}\nonumber\\
&+P^{-1}\frac{\partial P}{\partial \eta_{k}}(S^{-1}-P^{-1})\frac{\partial P}{\partial \eta_{l}}S^{-1}\nonumber\\
&+P^{-1}\frac{\partial P}{\partial\eta_{k}}P^{-1}\frac{\partial P}{\partial \eta_{l}}(S^{-1}-P^{-1})\overset{a.s.}\to 0.
\end{align}

Note that $\|\hat{\theta}^{\LS}\|_{2}=O_{p}(1)$, $\|\hat{\theta}^{\LS}-\theta_{0}\|_{2}=O_{p}(1/\sqrt{N})$, $\|S^{-1}-P^{-1}\|_{F}=O_{p}(1/N)$, $\left\|\frac{\partial S^{-1}}{\partial\eta_{k}}-\frac{\partial P^{-1}}{\partial\eta_{k}}\right\|_{F}=O_{p}(1/N)$ and  $\left\|\frac{\partial^2 S^{-1}}{\partial\eta_{k}\partial\eta_{l}}-\frac{\partial^2 P^{-1}}{\partial\eta_{k}\partial\eta_{l}}\right\|_{F}=O_{p}(1/N)$.
Thus $\Psi_{1,b}$ and $\Psi_{2,b}$ converge to zero almost surely and uniformly in $\overline{\Omega}_{2}$, which implies that $\frac{\partial^2 \overline{\mathcal{F}_{\EB}}}{\partial \eta\partial\eta^{T}}$ converges to $\frac{\partial^2 W_{b}}{\partial \eta\partial\eta^{T}}$ in probability and uniformly in $\overline{\Omega}_{2}$. From Lemma \ref{lemma:convergence in probability}, $\left.\frac{\partial^2 \overline{\mathcal{F}_{\EB}}}{\partial \eta\partial\eta^{T}}\right|_{\overline{\eta}_{N}}$ converges to $\left.\frac{\partial^2 W_{b}}{\partial \eta\partial\eta^{T}}\right|_{\eta^{*}_{b}}$ in probability for any sequence $\overline{\eta}_{N}$ such that ${\plim}_{N\to\infty}\overline{\eta}_{N}=\eta^{*}_{b}$.

The second step is to show that
\begin{align}
A_{b}(\eta^{*}_{b})\triangleq{\plim}_{N\to\infty}\E\left.\left[\frac{\partial^2 \overline{\mathcal{F}_{\EB}}}{\partial \eta\partial\eta^{T}}\right]\right|_{\eta_{b}^{*}}
=\left.\frac{\partial^2 W_{b}}{\partial \eta\partial\eta^{T}}\right|_{\eta^{*}_{b}}.
\end{align}
Under the assumption of $V\sim \mathcal{N}({\0},\sigma^2I_{N})$, we have
\begin{align}
\hat{\theta}^{\LS}\sim\mathcal{N}(\Phi\theta_{0},\sigma^2(\Phi^{T}\Phi)^{-1}).
\end{align}
Based on Lemma \ref{lemma:mean of quartic forms}, the $(k,l)$th element of $A_{b}(\eta^{*}_{b})$ is
\begin{align}
&A_{b}(\eta^{*}_{b})_{k,l}\nonumber\\
=&{\plim}_{N\to\infty} \left[\Tr\left(\sigma^2(\Phi^{T}\Phi)^{-1}\frac{\partial^2 S^{-1}}{\partial\eta_{k}\partial\eta_{l}}\right)
+\theta_{0}^{T}\frac{\partial^2 S^{-1}}{\partial\eta_{k}\partial\eta_{l}}\theta_{0}\right.\nonumber\\
&+\left.\left.\Tr\left(\frac{\partial S^{-1}}{\partial\eta_{l}}\frac{\partial P}{\partial\eta_{k}}\right)+\Tr\left(S^{-1}\frac{\partial^2 P}{\partial\eta_{k}\partial\eta_{l}}\right)\right]\right|_{\eta^{*}_{b}}.
\end{align}
Under the assumption that as $N\to\infty$, $(\Phi^{T}\Phi)/N \overset{a.s.}\to \Sigma\succ0$, it can be derived that $(\Phi^{T}\Phi)^{-1}\overset{a.s.}\to 0$, $S^{-1} \overset{a.s.}\to P^{-1}$, $\frac{\partial S^{-1}}{\partial \eta_{k}} \overset{a.s.}\to \frac{\partial P^{-1}}{\partial \eta_{k}}$ and $\frac{\partial^2 S^{-1}}{\partial\eta_{k}\partial\eta_{l}} \overset{a.s.}\to \frac{\partial^2 P^{-1}}{\partial\eta_{k}\partial\eta_{l}}$ as $N\to\infty$.
Then we have
\begin{align}
A_{b}(\eta^{*}_{b})_{k,l}=&\left\{\theta_{0}^{T}\frac{\partial^2 P^{-1}}{\partial\eta_{k}\partial\eta_{l}}\theta_{0}
                     +\Tr\left(\frac{\partial P^{-1}}{\partial \eta_{l}}\frac{\partial P}{\partial \eta_{k}}\right)\right.\nonumber\\
                     &+\left.\left.\Tr\left(P^{-1}\frac{\partial^2 P}{\partial\eta_{k}\partial\eta_{l}} \right)\right\}\right|_{\eta^{*}_{b}},
\end{align}
which is exactly equal to \eqref{eq:2nd derivatives of Wb} at $\eta_{b}^{*}$.

\item \underline{assumption \ref{asp6 in lemma:convergence of 1st derivative} in Lemma \ref{lemma:asymptotic normality of a consistent root}}

In this part, we show that as $N\to\infty$,
\begin{align}
\sqrt{N}\left.\frac{\partial \overline{\mathcal{F}_{\EB}}}{\partial \eta}\right|_{\eta_{b}^{*}} \overset{d.}\to \mathcal{N}({\0},B_{b}(\eta_{b}^{*})),
\end{align}
where
\begin{align}
B_{b}(\eta^{*}_{b})\triangleq &\plim_{N\to\infty}N\E\left(\left.\frac{\partial \overline{\mathcal{F}_{\EB}}}{\partial \eta}\right|_{\eta^{*}_{b}}\times \left.{\frac{\partial \overline{\mathcal{F}_{\EB}}}{\partial \eta^{T}}}\right|_{\eta^{*}_{b}}\right).
\end{align}

Our proof is made up of two steps.

The first step is to show that $\sqrt{N}\left.\frac{\partial \overline{\mathcal{F}_{\EB}}}{\partial \eta}\right|_{\eta_{b}^{*}}$ converges in distribution to a Gaussian distributed random vector with zero mean and the $(k,l)$th element of the limiting covariance matrix is
\begin{align}
\left.4\sigma^2\theta_{0}^{T}\frac{\partial P^{-1}}{\partial\eta_{k}}\Sigma^{-1}\frac{\partial P^{-1}}{\partial\eta_{l}}\theta_{0}\right|_{\eta_{b}^{*}}.
\end{align}

The $k$th elements of $\frac{\partial \overline{\mathcal{F}_{\EB}}}{\partial \eta}$ and $\frac{\partial W_{b}}{\partial \eta}$ can be written as
\begin{align}	
\frac{\partial \overline{\mathcal{F}_{\EB}}}{\partial \eta_{k}}=&(\hat{\theta}^{\LS})^{T}\frac{\partial S^{-1}}{\partial \eta_{k}}\hat{\theta}^{\LS}+\Tr\left(S^{-1}\frac{\partial P}{\partial \eta_{k}}\right)\\
\frac{\partial W_{b}}{\partial \eta_{k}}=&\theta_{0}^{T}\frac{\partial P^{-1}}{\partial \eta_{k}}\theta_{0}+\Tr\left(P^{-1}\frac{\partial P}{\partial \eta_{k}}\right).
\end{align}
Since $\left.\frac{\partial W_{b}}{\partial \eta_{k}}\right|_{\eta_{b}^{*}}=0$, it leads to
\begin{align}
\sqrt{N}\left.\frac{\partial \overline{\mathcal{F}_{\EB}}}{\partial \eta_{k}}\right|_{\eta_{b}^{*}}
=&\sqrt{N}\left.\left(\frac{\partial \overline{\mathcal{F}_{\EB}}}{\partial \eta_{k}}-\frac{\partial W_{b}}{\partial \eta_{k}}\right)\right|_{\eta_{b}^{*}}\nonumber\\
=&\sqrt{N}\left.\left[{\Upsilon_{1,b}}+{\Upsilon_{2,b}}\right]\right|_{\eta_{b}^{*}},
\end{align}
where
\begin{align}
\Upsilon_{1,b}
=&\theta_{0}^{T}\left(\frac{\partial S^{-1}}{\partial \eta_{k}}-\frac{\partial P^{-1}}{\partial \eta_{k}}\right)\hat{\theta}^{\LS}
+\Tr\left[(S^{-1}-P^{-1})\frac{\partial P}{\partial \eta_{k}}\right]\\
\Upsilon_{2,b}
=&(\hat{\theta}^{\LS}-\theta_{0})^{T}\frac{\partial S^{-1}}{\partial \eta_{k}}\hat{\theta}^{\LS}+\theta_{0}^{T}\frac{\partial P^{-1}}{\partial \eta_{k}}(\hat{\theta}^{\LS}-\theta_{0}).
\end{align}
\begin{enumerate}
\item For $\sqrt{N}\Upsilon_{1,b}|_{\eta_{b}^{*}}$, applying Lemma \ref{lemma:Op and convergence in probability} with $X_{N}=S^{-1}-P^{-1}$, $a_{N}=\frac{1}{N}$ and $w_{N}=\sqrt{N}$, we have
    \begin{align}\label{eq:convergence in probability  of sqrtN_S_inv_P_inv}
    \sqrt{N}(S^{-1}-P^{-1}) \overset{p}\to 0
    \end{align}
     as $N\to\infty$. Similarly, we can prove that as $N\to\infty$,
     \begin{align}\label{eq:convergence in probability of sqrtN_1st_derivative_S_inv_P_inv}
     \sqrt{N}\left(\frac{\partial S^{-1}}{\partial \eta_{k}}-\frac{\partial P^{-1}}{\partial \eta_{k}}\right) \overset{p}\to 0
     \end{align}
     with $X_{N}=\frac{\partial S^{-1}}{\partial \eta_{k}}-\frac{\partial P^{-1}}{\partial \eta_{k}}$, $a_{N}=\frac{1}{N}$ and $w_{N}=\sqrt{N}$ by Lemma \ref{lemma:Op and convergence in probability}.
      Note that as $N\to\infty$, $\frac{\partial S^{-1}}{\partial \eta_{k}}\overset{p}\to \frac{\partial P^{-1}}{\partial \eta_{k}}$, $\hat{\theta}^{\LS} \overset{p}\to \theta_{0}$, \eqref{eq:convergence in probability  of sqrtN_S_inv_P_inv} and \eqref{eq:convergence in probability of sqrtN_1st_derivative_S_inv_P_inv} will not change with the value of $\eta$. Thus, according to the sum and product rules of convergence in probability, it can be seen that as $N\to\infty$,
     \begin{align}
     \sqrt{N}\Upsilon_{1,b}|_{\eta_{b}^{*}}\overset{p}\to 0.
     \end{align}

\item For $\sqrt{N}\Upsilon_{2,b}|_{\eta_{b}^{*}}$, let us investigate the limiting in distribution of $\sqrt{N}(\hat{\theta}^{\LS}-\theta_{0})$ firstly, which can be rewritten as
\begin{align}
\sqrt{N}(\hat{\theta}^{\LS}-\theta_{0})=&\sqrt{N}(\Phi^{T}\Phi)^{-1}\Phi^{T}V\nonumber\\
=&[N(\Phi^{T}\Phi)^{-1}] \left[\sqrt{N}\frac{\Phi^{T}V}{N}\right].
\end{align}
Since as $N\to\infty$, $N(\Phi^{T}\Phi)^{-1}\overset{p}\to\Sigma^{-1}$ and $\sqrt{N}\frac{\Phi^{T}V}{N}\overset{d}\to\mathcal{N}(0,\sigma^2\Sigma)$, which can be proved by CLT, then it follows that
\begin{align}\label{eq:limiting distribution of sqrtN_theta_LS_theta0}
\sqrt{N}(\hat{\theta}^{\LS}-\theta_{0})\overset{d}\to\mathcal{N}(0,\sigma^2\Sigma^{-1}).
\end{align}
Note that as $N\to\infty$, $\frac{\partial S^{-1}}{\partial \eta_{k}}\overset{p}\to \frac{\partial P^{-1}}{\partial \eta_{k}}$, $\hat{\theta}^{\LS} \overset{p}\to \theta_{0}$, and \eqref{eq:limiting distribution of sqrtN_theta_LS_theta0} will not change with the value of $\eta$. Thus, according to the product rule of the limiting in distribution, we have as $N\to\infty$,
\begin{align}
\sqrt{N}\Upsilon_{2,b}|_{\eta_{b}^{*}} \overset{d}\to \mathcal{N}(0,\left.4\sigma^2\theta_{0}^{T}\frac{\partial P^{-1}}{\partial\eta_{k}}\Sigma^{-1}\frac{\partial P^{-1}}{\partial\eta_{k}}\theta_{0}\right|_{\eta_{b}^{*}}).
\end{align}
\end{enumerate}

Then we come to
\begin{align}
\sqrt{N}\left.\frac{\partial \overline{\mathcal{F}_{\EB}}}{\partial \eta_{k}}\right|_{\eta_{b}^{*}}\overset{d}\to\mathcal{N}(0,\left.4\sigma^2\theta_{0}^{T}\frac{\partial P^{-1}}{\partial\eta_{k}}\Sigma^{-1}\frac{\partial P^{-1}}{\partial\eta_{k}}\theta_{0}\right|_{\eta_{b}^{*}}).
\end{align}
Therefore, $\sqrt{N}\left.\frac{\partial \overline{\mathcal{F}_{\EB}}}{\partial \eta}\right|_{\eta_{b}^{*}}$ converges in distribution to a Gaussian distributed random vector with zero mean and the $(k,l)$th element of the limiting covariance matrix is
\begin{align}\label{eq:limiting covariance of 1st derivative of Feb}
\left.4\sigma^2\theta_{0}^{T}\frac{\partial P^{-1}}{\partial\eta_{k}}\Sigma^{-1}\frac{\partial P^{-1}}{\partial\eta_{l}}\theta_{0}\right|_{\eta_{b}^{*}}.
\end{align}

The second step is to show that the $(k,l)$th element of
\begin{align}
B_{b}(\eta^{*}_{b})\triangleq &\plim_{N\to\infty}N\E\left(\left.\frac{\partial \overline{\mathcal{F}_{\EB}}}{\partial \eta}\right|_{\eta^{*}_{b}}\times \left.{\frac{\partial \overline{\mathcal{F}_{\EB}}}{\partial \eta^{T}}}\right|_{\eta^{*}_{b}}\right)
\end{align}
equals \eqref{eq:limiting covariance of 1st derivative of Feb}.
By Lemma \ref{lemma:mean of quartic forms}, the $(k,l)$th element of $B(\eta^{*}_{b})$ can be rewritten as
\begin{align}
&B_{b}(\eta^{*}_{b})_{k,l}\nonumber\\
=&{\plim}_{N\to\infty} N\E\left[\left.\frac{\partial \overline{\mathcal{F}_{\EB}}}{\partial \eta_{k}}\right|_{\eta^{*}_{b}}\left.\frac{\partial \overline{\mathcal{F}_{\EB}}}{\partial \eta_{l}}\right|_{\eta^{*}_{b}}\right]\nonumber\\
=&\left\{4\sigma^2\theta_{0}^{T}\frac{\partial P^{-1}}{\partial \eta_{k}}\Sigma^{-1}\frac{\partial P^{-1}}{\partial \eta_{l}}\theta_{0}\right.\nonumber\\
&+\plim_{N\to\infty}\sqrt{N}\left[\theta_{0}^{T}\frac{\partial S^{-1}}{\partial \eta_{k}}\theta_{0}+ \Tr\left(S^{-1}\frac{\partial P}{\partial \eta_{k}}\right)\right]\nonumber\\
&\left.\left.\sqrt{N}\left[\theta_{0}^{T}\frac{\partial S^{-1}}{\partial \eta_{l}}\theta_{0}+ \Tr\left(S^{-1}\frac{\partial P}{\partial \eta_{l}}\right)\right]\right\}\right|_{\eta^{*}_{b}}.
\end{align}
Based on \eqref{eq:convergence in probability  of sqrtN_S_inv_P_inv} and \eqref{eq:convergence in probability of sqrtN_1st_derivative_S_inv_P_inv}, we can prove that
\begin{align}
&\plim_{N\to\infty}\left.\sqrt{N}\left[\theta_{0}^{T}\frac{\partial S^{-1}}{\partial \eta_{k}}\theta_{0}+ \Tr\left(S^{-1}\frac{\partial P}{\partial \eta_{k}}\right)\right]\right|_{\eta^{*}_{b}}\nonumber\\
=&\plim_{N\to\infty}\left.\sqrt{N}\left[\theta_{0}^{T}\frac{\partial S^{-1}}{\partial \eta_{k}}\theta_{0}+ \Tr\left(S^{-1}\frac{\partial P}{\partial \eta_{k}}\right)-\frac{\partial W_{b}}{\partial \eta_{k}}\right]\right|_{\eta^{*}_{b}}\nonumber\\
=&\plim_{N\to\infty}\left. \sqrt{N}\theta_{0}^{T}\left(\frac{\partial S^{-1}}{\partial \eta_{k}}-\frac{\partial P^{-1}}{\partial \eta_{k}}\right)\theta_{0}\right.\nonumber\\
&+\left.\sqrt{N}\Tr\left[(S^{-1}-P^{-1})\frac{\partial P}{\partial\eta_{k}}\right]
\right|_{\eta^{*}_{b}}=0.
\end{align}
Then we have
\begin{align}
B_{b}(\eta^{*}_{b})_{k,l}=\left.4\sigma^2\theta_{0}^{T}\frac{\partial P^{-1}}{\partial \eta_{k}}\Sigma^{-1}\frac{\partial P^{-1}}{\partial \eta_{l}}\theta_{0}\right|_{\eta^{*}_{b}}.
\end{align}
\end{itemize}

Finally, we apply Lemma \ref{lemma:asymptotic normality of a consistent root} with $M_{N}=N\overline{\mathcal{F}_{\EB}}$ and $\eta^{*}=\eta^{*}_{b}$ to prove the asymptotic normality of $\hat{\eta}_{\EB}-\eta^{*}_{b}$.

\subsection{Proof of Theorem \ref{thm:asymptotic normality of eta_sy difference}}

The asymptotic normality of $\hat{\eta}_{\Sy}-\eta_{y}^{*}$ can be shown through similar thoughts with $M_{N}=N\overline{\mathcal{F}_{\Sy}}$ and $\eta^{*}=\eta^{*}_{y}$.

\begin{itemize}
\item \underline{assumption \ref{asp1 in lemma:Euclidean space}, \ref{asp2 in lemma:measurability of F} and \ref{asp4 in lemma:continuity of 2nd derivative} in Lemma \ref{lemma:asymptotic normality of a consistent root}}

First of all, we show that $N\overline{\mathcal{F}_{\Sy}}(Y,\eta)$ is a measurable function of $Y$ for all $\eta\in\Omega$. Recall that
\begin{align}
\overline{\mathcal{F}_{\Sy}}
=&N\sigma^4[(\hat{\theta}^{\LS})^{T}S^{-T}(\Phi^{T}\Phi)^{-1}S^{-1}\hat{\theta}^{\LS}\nonumber\\
&-2\Tr((\Phi^{T}\Phi)^{-1}S^{-1})]\nonumber\\
=&N[\sigma^4Y^{T}Q^{-T}\Phi(\Phi^{T}\Phi)^{-1}\Phi^{T}Q^{-1}Y\nonumber\\
&+2\sigma^2\Tr((\Phi^{T}\Phi+\sigma^2 P^{-1})^{-1}\Phi^{T}\Phi-I_{n})].
\end{align}
It can be noticed that $N\overline{\mathcal{F}_{\Sy}}(Y,\eta)$ is a continuous function of $Y$ for all $\eta$ in $\Omega$, which indicates that $\forall\eta\in\Omega$, $N\overline{\mathcal{F}_{\Sy}}(Y,\eta)$ is a measurable function of $Y$.

Then we show that $\frac{\partial N\overline{\mathcal{F}_{\Sy}}}{\partial\eta}$ exists and is continuous in an open neighbourhood of $\eta^{*}_{y}$, and $\frac{\partial^2 N\overline{\mathcal{F}_{\Sy}}}{\partial\eta\partial\eta^{T}}$ exists and is continuous in an open and convex neighbourhood of $\eta_{y}^{*}$.

For common kernel structures, like SS \eqref{eq:SS kernel}, DC \eqref{eq:DC kernel} and TC \eqref{eq:TC kernel}, $P(\eta)$ is a continuous, differentiable and second-order differentiable function of $\eta$ in $\Omega$. Meanwhile, the first-order derivative and the second-order derivative of $P(\eta)$ with respect to $\eta$ are both continuous for all $\eta\in\Omega$. Then under Assumption \ref{asp:interior points eta_eb_star eta_sy_star}, it can be derived that there exists an open neighbourhood of $\eta^{*}_{y}$ such that $\frac{\partial N\overline{\mathcal{F}_{\Sy}}}{\partial\eta}$ exists and is continuous. There also exists an open and convex neighbourhood of $\eta_{y}^{*}$, in which $\frac{\partial^2 N\overline{\mathcal{F}_{\Sy}}}{\partial\eta\partial\eta^{T}}$ exists and is continuous.

\item \underline{assumption \ref{asp3 in lemma:convergence of F} in Lemma \ref{lemma:asymptotic normality of a consistent root}}

In this part, we prove that $\overline{\mathcal{F}_{\Sy}}(\eta)$ converges to $W_{y}(\eta)$ in probability and uniformly in one neighbourhood of $\eta_{y}^{*}$.

As mentioned in \eqref{eq:difference of Fsy and Wy}, \eqref{eq:term1 difference of Fsy and Wy} and \eqref{eq:term2 difference of Fsy and Wy}, $\overline{\mathcal{F}_{\Sy}}-W_{y}$ is computed with two parts: $D_{1,y}$ and $\Tr(D_{2,y})$. Recall that
\begin{align}
D_{1,y}	=&\sigma^4(\hat{\theta}^{\LS}-\theta_{0})^{T}S^{-T}N(\Phi^{T}\Phi)^{-1}S^{-1}\hat{\theta}^{\LS}\nonumber\\	&+\sigma^4\theta_{0}^{T}(S^{-1}-P^{-1})^{T}N(\Phi^{T}\Phi)^{-1}S^{-1}\hat{\theta}^{\LS}\nonumber\\	&+\sigma^4\theta_{0}^{T}P^{-T}(N(\Phi^{T}\Phi)^{-1}-\Sigma^{-1})S^{-1}\hat{\theta}^{\LS}\nonumber\\
&+\sigma^4\theta_{0}^{T}P^{-T}\Sigma^{-1}(S^{-1}-P^{-1})\hat{\theta}^{\LS}\nonumber\\
&+\sigma^4\theta_{0}^{T}P^{-T}\Sigma^{-1}P^{-1}(\hat{\theta}^{\LS}-\theta_{0})\\
D_{2,y}
=&2\sigma^4(\Sigma^{-1}-N(\Phi^{T}\Phi)^{-1})P^{-1}\nonumber\\
&+2\sigma^4N(\Phi^{T}\Phi)^{-1}(P^{-1}-S^{-1}).
\end{align}
Based on Assumption \ref{asp:interior points eta_eb_star eta_sy_star}, there exists $\overline{\Omega}_{3}(\eta_{y}^{*})\subset\Omega$ such that $0<d_{3}\leq\|P(\eta)\|_{F}\leq d_{4}<\infty$ and $\|S^{-1}\|_{F}\leq\|P^{-1}\|_{F}\leq 1/d_{3}$ for all $\eta\in\overline{\Omega}_{3}$. Noting that $N(\Phi^{T}\Phi)^{-1}\overset{a.s.}\to\Sigma$, $\hat{\theta}^{\LS} \overset{a.s.}\to \theta_{0}$, $S^{-1} \overset{a.s.}\to P^{-1}$ as $N\to\infty$, and $\|\hat{\theta}^{\LS}-\theta_{0}\|_{2}=O_{p}(1/\sqrt{N})$, $\|\hat{\theta}^{\LS}\|_{2}=O_{p}(1)$, $\|S^{-1}-P^{-1}\|_{F}=O_{p}(1/N)$ and $\|N(\Phi^{T}\Phi)^{-1}-\Sigma^{-1}\|_{F}=O_{p}(\delta_{N})$, we can show that each term of $D_{1,y}$ and $D_{2,y}$ converges to zero in probability and uniformly for any $\eta$ in $\overline{\Omega_{3}}$. Thus, as $N\to\infty$, $\overline{\mathcal{F}_{\Sy}}$ converges to $W_{y}$ in probability and uniformly $\forall\eta\in\overline{\Omega}_{3}$.

Under Assumption \ref{asp:2} and \ref{asp:interior points eta_eb_star eta_sy_star}, we can show that $W_{y}$ attains a strict local minimum at $\eta_{y}^{*}$.

\item \underline{assumption \ref{asp5 in lemma:convergence of 2nd derivative} in Lemma \ref{lemma:asymptotic normality of a consistent root}}

Our goal in this part is to prove that $\left.\frac{\partial^2 \overline{\mathcal{F}_{\Sy}}}{\partial \eta\partial\eta^{T}}\right|_{\overline{\eta}_{N}}$ converges to $C_{y}(\eta_{y}^{*})$ in probability for any sequence $\overline{\eta}_{N}$ such that ${\lim}_{N\to\infty}\overline{\eta}_{N}=\eta^{*}_{y}$ in probability, where
\begin{align}
C_{y}(\eta^{*}_{y})\triangleq{\plim}_{N\to\infty}\E\left.\left[\frac{\partial^2 \overline{\mathcal{F}_{\Sy}}}{\partial \eta\partial\eta^{T}}\right]\right|_{\eta_{y}^{*}}.
\end{align}

Detailed procedure consists of two steps.

The first step is to prove that $\left.\frac{\partial^2 \overline{\mathcal{F}_{\Sy}}}{\partial \eta\partial\eta^{T}}\right|_{\overline{\eta}_{N}}$ converges to $\left.\frac{\partial^2 W_{y}}{\partial \eta\partial\eta^{T}}\right|_{\eta^{*}_{y}}$ in probability for any sequence $\overline{\eta}_{N}$ such that ${\lim}_{N\to\infty}\overline{\eta}_{N}=\eta^{*}_{y}$ in probability.

The $(k,l)$th elements of Hessian matrices of $\overline{\mathcal{F}_{\Sy}}$ and $W_{y}$ are shown as follows, respectively,
\begin{align}
\frac{\partial^2 \overline{\mathcal{F}_{\Sy}}}{\partial\eta_{k}\partial\eta_{l}}
=&2\sigma^4(\hat{\theta}^{\LS})^{T}\frac{\partial S^{-T}}{\partial\eta_{l}}N(\Phi^{T}\Phi)^{-1}\frac{\partial S^{-1}}{\partial\eta_{k}}\hat{\theta}^{\LS}\nonumber\\
&+2\sigma^4(\hat{\theta}^{\LS})^{T}S^{-T}N(\Phi^{T}\Phi)^{-1}\frac{\partial^2 S^{-1}}{\partial\eta_{k}\partial\eta_{l}}\hat{\theta}^{\LS}\nonumber\\
&-2\sigma^4\Tr\left(N(\Phi^{T}\Phi)^{-1}\frac{\partial^2 S^{-1}}{\partial\eta_{k}\partial\eta_{l}}\right)
\end{align}
\begin{align}
\label{eq:2nd derivatives of Wy}
\frac{\partial^2 W_{y}}{\partial\eta_{k}\partial\eta_{l}}
=&2\sigma^4\theta_{0}^{T}\frac{\partial P^{-T}}{\partial\eta_{l}}\Sigma^{-1}\frac{\partial P^{-1}}{\partial\eta_{k}}\theta_{0}\nonumber\\
&+2\sigma^4\theta_{0}^{T}P^{-T}\Sigma^{-1}\frac{\partial^2 P^{-1}}{\partial\eta_{k}\partial\eta_{l}}\theta_{0}\nonumber\\
&-2\sigma^4\Tr\left(\Sigma^{-1}\frac{\partial^2 P^{-1}}{\partial\eta_{k}\partial\eta_{l}}\right).
\end{align}
Then the $(k,l)$th element of the difference between $\frac{\partial^2 \overline{\mathcal{F}_{\Sy}}}{\partial \eta\partial\eta^{T}}$ and $\frac{\partial^2 W_{y}}{\partial \eta\partial\eta^{T}}$ can be represented as
\begin{align}
\frac{\partial^2 \overline{\mathcal{F}_{\Sy}}}{\partial\eta_{k}\partial\eta_{l}}-\frac{\partial^2 W_{y}}{\partial\eta_{k}\partial\eta_{l}}
=&\Psi_{1,y}+\Tr(\Psi_{2,y}),
\end{align}
where
\begin{align}
\Psi_{1,y}=&2\sigma^4(\hat{\theta}^{\LS}-\theta_{0})^{T}\frac{\partial S^{-T}}{\partial\eta_{l}}N(\Phi^{T}\Phi)^{-1}\frac{\partial S^{-1}}{\partial\eta_{k}}\hat{\theta}^{\LS}\nonumber\\
&+2\sigma^4(\theta_{0})^{T}\left(\frac{\partial S^{-T}}{\partial\eta_{l}}-\frac{\partial P^{-T}}{\partial\eta_{l}}\right)N(\Phi^{T}\Phi)^{-1}\frac{\partial S^{-1}}{\partial\eta_{k}}\hat{\theta}^{\LS}\nonumber\\
&+2\sigma^4(\theta_{0})^{T}\frac{\partial P^{-T}}{\partial\eta_{l}}(N(\Phi^{T}\Phi)^{-1}-\Sigma^{-1})\frac{\partial S^{-1}}{\partial\eta_{k}}\hat{\theta}^{\LS}\nonumber\\
&+2\sigma^4(\theta_{0})^{T}\frac{\partial P^{-T}}{\partial\eta_{l}}\Sigma^{-1}\left(\frac{\partial S^{-1}}{\partial\eta_{k}}-\frac{\partial P^{-1}}{\partial\eta_{k}}\right)\hat{\theta}^{\LS}\nonumber\\
&+2\sigma^4(\theta_{0})^{T}\frac{\partial P^{-T}}{\partial\eta_{l}}\Sigma^{-1}\frac{\partial P^{-1}}{\partial\eta_{k}}(\hat{\theta}^{\LS}-\theta_{0})\nonumber\\
&+2\sigma^4(\hat{\theta}^{\LS}-\theta_{0})^{T}S^{-T}N(\Phi^{T}\Phi)^{-1}\frac{\partial^2 S^{-1}}{\partial\eta_{k}\partial\eta_{l}}\hat{\theta}^{\LS}\nonumber\\
&+2\sigma^4(\theta_{0})^{T}(S^{-1}-P^{-1})^{T}N(\Phi^{T}\Phi)^{-1}\frac{\partial^2 S^{-1}}{\partial\eta_{k}\partial\eta_{l}}\hat{\theta}^{\LS}\nonumber\\
&+2\sigma^4(\theta_{0})^{T}P^{-T}(N(\Phi^{T}\Phi)^{-1}-\Sigma^{-1})\frac{\partial^2 S^{-1}}{\partial\eta_{k}\partial\eta_{l}}\hat{\theta}^{\LS}\nonumber\\
&+2\sigma^4(\theta_{0})^{T}P^{-T}\Sigma^{-1}\left(\frac{\partial^2 S^{-1}}{\partial\eta_{k}\partial\eta_{l}}-\frac{\partial^2 P^{-1}}{\partial\eta_{k}\partial\eta_{l}}\right)\hat{\theta}^{\LS}\nonumber\\
&+2\sigma^4(\theta_{0})^{T}P^{-T}\Sigma^{-1}\frac{\partial^2 P^{-1}}{\partial\eta_{k}\partial\eta_{l}}(\hat{\theta}^{\LS}-\theta_{0})\\
\Psi_{2,y}=&2\sigma^4(\Sigma^{-1}-N(\Phi^{T}\Phi)^{-1})\frac{\partial^2 P^{-1}}{\partial\eta_{k}\partial\eta_{l}}\nonumber\\
&+2\sigma^4N(\Phi^{T}\Phi)^{-1}\left(\frac{\partial^2 P^{-1}}{\partial\eta_{k}\partial\eta_{l}}-\frac{\partial^2 S^{-1}}{\partial\eta_{k}\partial\eta_{l}}\right).
\end{align}
Under Assumption \ref{asp:interior points eta_eb_star eta_sy_star}, there exists a neighborhood $\overline{\Omega}_{4}\subset\Omega$ of $\eta_{y}^{*}$ such that for any $k=1,\cdots,p$ and $l=1,\cdots,p$, $\frac{\partial^2 P}{\partial\eta_{k}\partial\eta_{l}}$, $\frac{\partial P}{\partial\eta_{k}}$ and $P$ are all bounded, which leads to that
$\frac{\partial^2 P^{-1}}{\partial\eta_{k}\partial\eta_{l}}$ and $\frac{\partial^2 S^{-1}}{\partial\eta_{k}\partial\eta_{l}}$ are both bounded $\forall\eta\in\overline{\Omega}_{4}$ with $k=1,\cdots,p$ and $l=1,\cdots,p$. As $N\to\infty$, we have $N(\Phi^{T}\Phi)^{-1} \overset{a.s.}\to \Sigma^{-1}$, $\hat{\theta}^{\LS}\overset{a.s.}\to\theta_{0}$, $S^{-1}\overset{a.s.}\to P^{-1}$, $\frac{\partial S^{-1}}{\partial\eta_{k}}\overset{a.s.}\to \frac{\partial P^{-1}}{\partial\eta_{k}}$ and $\frac{\partial^2 S^{-1}}{\partial\eta_{k}\partial\eta_{l}}\overset{a.s.}\to\frac{\partial^2 P^{-1}}{\partial\eta_{k}\partial\eta_{l}}$. Also note that $\|N(\Phi^{T}\Phi)^{-1}-\Sigma^{-1}\|_{F}=O_{p}(\delta_{N})$,  $\|\hat{\theta}^{\LS}-\theta_{0}\|_{2}=O_{p}(1/\sqrt{N})$, $\|S^{-1}-P^{-1}\|_{F}=O_{p}(1/N)$, $\left\|\frac{\partial S^{-1}}{\partial\eta_{k}}-\frac{\partial P^{-1}}{\partial\eta_{k}}\right\|_{F}=O_{p}(1/N)$ and  $\left\|\frac{\partial^2 S^{-1}}{\partial\eta_{k}\partial\eta_{l}}-\frac{\partial^2 P^{-1}}{\partial\eta_{k}\partial\eta_{l}}\right\|_{F}=O_{p}(1/N)$. So each term in $\Psi_{1,y}$ and $\Psi_{2,y}$ converges to zero in probability and uniformly, $\forall \eta\in\overline{\Omega}_{4}(\eta_{y}^{*})$.
Therefore, $\frac{\partial^2 \overline{\mathcal{F}_{\Sy}}}{\partial \eta\partial\eta^{T}}$ converges to $\frac{\partial^2 W_{y}}{\partial \eta\partial\eta^{T}}$ in probability and uniformly in $\overline{\Omega}_{4}(\eta_{y}^{*})$. From Lemma \ref{lemma:convergence in probability}, $\left.\frac{\partial^2 \overline{\mathcal{F}_{\Sy}}}{\partial \eta\partial\eta^{T}}\right|_{\overline{\eta}_{N}}$ converges to $\left.\frac{\partial^2 W_{y}}{\partial \eta\partial\eta^{T}}\right|_{\eta^{*}_{y}}$ in probability for any sequence $\overline{\eta}_{N}$ such that ${\lim}_{N\to\infty}\overline{\eta}_{N}=\eta^{*}_{y}$ in probability.

The second step is to show that
\begin{align}
C_{y}(\eta^{*}_{y})\triangleq{\plim}_{N\to\infty}\E\left.\left[\frac{\partial^2 \overline{\mathcal{F}_{\Sy}}}{\partial \eta\partial\eta^{T}}\right]\right|_{\eta_{y}^{*}}
=\left.\frac{\partial^2 W_{y}}{\partial \eta\partial\eta^{T}}\right|_{\eta^{*}_{y}}.
\end{align}

Since $\hat{\theta}^{\LS}\sim\mathcal{N}(\Phi\theta_{0},\sigma^2(\Phi^{T}\Phi)^{-1})$, we can apply Lemma \ref{lemma:mean of quartic forms} to obtain the $(k,l)$th element of $C(\eta^{*}_{y})$ as
\begin{align}
C_{y}(\eta^{*}_{y})_{k,l}=&\plim_{N\to\infty}2\sigma^6\Tr\left[\frac{\partial S^{-T}}{\partial\eta_{l}}N(\Phi^{T}\Phi)^{-1}\frac{\partial S^{-1}}{\partial\eta_{k}}(\Phi^{T}\Phi)^{-1}\right]\nonumber\\
&+2\sigma^4\theta_{0}^{T}\frac{\partial S^{-T}}{\partial\eta_{l}}N(\Phi^{T}\Phi)^{-1}\frac{\partial S^{-1}}{\partial\eta_{k}}\theta_{0}\nonumber\\
&+2\sigma^6\Tr\left[S^{-T}N(\Phi^{T}\Phi)^{-1}\frac{\partial^2 S^{-1}}{\partial\eta_{k}\partial\eta_{l}}(\Phi^{T}\Phi)^{-1}\right]\nonumber\\
&+2\sigma^4\theta_{0}^{T}S^{-T}N(\Phi^{T}\Phi)^{-1}\frac{\partial^2 S^{-1}}{\partial\eta_{k}\partial\eta_{l}}\theta_{0}\nonumber\\
&-\left.2\sigma^4\Tr\left(N(\Phi^{T}\Phi)^{-1}\frac{\partial^2 S^{-1}}{\partial\eta_{k}\partial\eta_{l}}\right)\right|_{\eta^{*}_{y}}
\end{align}
Since as $N\to\infty$, we have $\Phi^{T}\Phi\overset{a.s.}\to 0$, $N(\Phi^{T}\Phi)^{-1}\overset{a.s.}\to\Sigma^{-1}$, $S^{-1} \overset{a.s.}\to P^{-1}$, $\frac{\partial S^{-1}}{\partial \eta_{k}} \overset{a.s.}\to \frac{\partial P^{-1}}{\partial \eta_{k}}$ and $\frac{\partial^2 S^{-1}}{\partial\eta_{k}\partial\eta_{l}} \overset{a.s.}\to \frac{\partial^2 P^{-1}}{\partial\eta_{k}\partial\eta_{l}}$, $C_{y}(\eta_{y}^{*})$ can be reduced as
\begin{align}
C_{y}(\eta_{y}^{*})
=&\left\{2\sigma^4\theta_{0}^{T}\frac{\partial P^{-T}}{\partial\eta_{l}}{\Sigma}^{-1}\frac{\partial P^{-1}}{\partial\eta_{k}}\theta_{0}\right.\nonumber\\
&+2\sigma^4\theta_{0}^{T}P^{-T}{\Sigma}^{-1}\frac{\partial^2 P^{-1}}{\partial\eta_{k}\partial\eta_{l}}\theta_{0}\nonumber\\
&-\left.\left.2\sigma^4\Tr\left({\Sigma}^{-1}\frac{\partial^2 P^{-1}}{\partial\eta_{k}\partial\eta_{l}}\right)\right\}\right|_{\eta^{*}_{y}},
\end{align}
which is exactly equal to \eqref{eq:2nd derivatives of Wy} at $\eta_{y}^{*}$.

\item \underline{assumption \ref{asp6 in lemma:convergence of 1st derivative} in Lemma \ref{lemma:asymptotic normality of a consistent root}}

In this part, we aim to prove that $N\to\infty$,
\begin{align}
\sqrt{N}\left.\frac{\partial \overline{\mathcal{F}_{\Sy}}}{\partial \eta}\right|_{\eta_{y}^{*}} \overset{d.}\to \mathcal{N}({\0},D_{y}(\eta_{y}^{*})),
\end{align}
where
\begin{align}
D_{y}(\eta^{*}_{y})\triangleq &\plim_{N\to\infty}N\E\left(\left.\frac{\partial \overline{\mathcal{F}_{\Sy}}}{\partial \eta}\right|_{\eta^{*}_{y}}\times \left.{\frac{\partial \overline{\mathcal{F}_{\Sy}}}{\partial \eta^{T}}}\right|_{\eta^{*}_{y}}\right).
\end{align}

Our proof consists of two steps.

The first step is to show that $\sqrt{N}\left.\frac{\partial \overline{\mathcal{F}_{\Sy}}}{\partial \eta}\right|_{\eta_{y}^{*}} $ converges in distribution to a Gaussian distributed random vector with zero mean and the $(k,l)$th element of the limiting covariance matrix is
\begin{align}\label{eq:limiting covariance of 1st derivative of Fsy}
&4\sigma^{10}\left\{\theta_{0}^{T}\left[P^{-1}{\Sigma}^{-1}\frac{\partial P^{-1}}{\partial\eta_{k}} + \frac{\partial P^{-1}}{\partial\eta_{k}}{\Sigma}^{-1}P^{-1} \right]{\Sigma}^{-1}\right.\nonumber\\
&\left.\left.\left[P^{-1}{\Sigma}^{-1}\frac{\partial P^{-1}}{\partial\eta_{l}} + \frac{\partial P^{-1}}{\partial\eta_{l}}{\Sigma}^{-1}P^{-1} \right]\theta_{0}\right\}\right|_{\eta^{*}_{y}}.
\end{align}
The $k$th elements of $\frac{\partial\overline{\mathcal{F}_{\Sy}}}{\partial\eta}$ and $\frac{\partial W_{y}}{\partial\eta}$ can be written as
\begin{align}
\frac{\partial\overline{\mathcal{F}_{\Sy}}}{\partial\eta_{k}}
=&2\sigma^4(\hat{\theta}^{\LS})^{T}S^{-T}N(\Phi^{T}\Phi)^{-1}\frac{\partial S^{-1}}{\partial\eta_{k}}\hat{\theta}^{\LS}\nonumber\\
&-2\sigma^4\Tr\left(N(\Phi^{T}\Phi)^{-1}\frac{\partial S^{-1}}{\partial\eta_{k}}\right)\\
\frac{\partial W_{y}}{\partial\eta_{k}}
=&2\sigma^4\theta_{0}^{T}P^{-T}\Sigma^{-1}\frac{\partial P^{-1}}{\partial\eta_{k}}\theta_{0}\nonumber\\
&-2\sigma^4\Tr\left(\Sigma^{-1}\frac{\partial P^{-1}}{\partial\eta_{k}}\right).
\end{align}
Since $\left.\frac{\partial W_{y}}{\partial\eta_{k}}\right|_{\eta_{y}^{*}}=0$, we have
\begin{align}
\sqrt{N}\left.\frac{\partial \overline{\mathcal{F}_{\Sy}}}{\partial \eta}\right|_{\eta_{y}^{*}}
=&\sqrt{N}\left.\left[\frac{\partial\overline{\mathcal{F}_{\Sy}}}{\partial\eta_{k}}-\frac{\partial W_{y}}{\partial\eta_{k}}\right]\right|_{\eta_{y}^{*}}\nonumber\\
=&\sqrt{N}(\Upsilon_{1,y}|_{\eta_{y}^{*}}+\Upsilon_{2,y}|_{\eta_{y}^{*}}),
\end{align}
where
\begin{align}
\Upsilon_{1,y}
=&2\sigma^4\theta_{0}^{T}(S^{-1}-P^{-1})^{T}N(\Phi^{T}\Phi)^{-1}\frac{\partial S^{-1}}{\partial\eta_{k}}\hat{\theta}^{\LS}\nonumber\\
&+2\sigma^4\theta_{0}^{T}P^{-T}(N(\Phi^{T}\Phi)^{-1}-\Sigma^{-1})\frac{\partial S^{-1}}{\partial\eta_{k}}\hat{\theta}^{\LS}\nonumber\\
&+2\sigma^4\theta_{0}^{T}P^{-T}\Sigma^{-1}\left(\frac{\partial S^{-1}}{\partial\eta_{k}}-\frac{\partial P^{-1}}{\partial\eta_{k}}\right)\hat{\theta}^{\LS}\nonumber\\
&+2\sigma^4\Tr\left[(\Sigma^{-1}-N(\Phi^{T}\Phi)^{-1})\frac{\partial P^{-1}}{\partial\eta_{k}}\right]\nonumber\\
&+2\sigma^4\Tr\left[N(\Phi^{T}\Phi)^{-1}\left(\frac{\partial P^{-1}}{\partial\eta_{k}}-\frac{\partial S^{-1}}{\partial\eta_{k}}\right)\right]\\
\Upsilon_{2,y}
=&2\sigma^4(\hat{\theta}^{\LS}-\theta_{0})^{T}S^{-T}N(\Phi^{T}\Phi)^{-1}\frac{\partial S^{-1}}{\partial\eta_{k}}\hat{\theta}^{\LS}\nonumber\\
&+2\sigma^4\theta_{0}^{T}P^{-T}\Sigma^{-1}\frac{\partial P^{-1}}{\partial\eta_{k}}(\hat{\theta}^{\LS}-\theta_{0}).
\end{align}
\begin{enumerate}
\item For $\sqrt{N}\Upsilon_{1,y}|_{\eta_{y}^{*}}$, if we define that $w_{N}=\frac{1}{\delta_{N}\sqrt{N}}$, from $\delta_{N}=o(1/\sqrt{N})$ as $N\to\infty$, it can be derived that
    \begin{align}
    \lim_{N\to\infty}\frac{1}{w_{N}}=\lim_{N\to\infty}\frac{\delta_{N}}{1/\sqrt{N}}=0,
    \end{align}
    namely $w_{N}\to\infty$ as $N\to\infty$.
    Thus we can use Lemma \ref{lemma:Op and convergence in probability} with $X_{N}=N(\Phi^{T}\Phi)^{-1}-\Sigma^{-1}$, $a_{N}=\delta_{N}$ and $w_{N}=\frac{1}{\delta_{N}\sqrt{N}}$ to obtain
    \begin{align}\label{eq:limiting in probability of sqrtN_PP_Sigma_inv}
    \sqrt{N}[N(\Phi^{T}\Phi)^{-1}-\Sigma^{-1}]\overset{p}\to 0.
    \end{align}
    As $N\to\infty$, $N(\Phi^{T}\Phi)^{-1}\overset{p}\to\Sigma^{-1}$, $S^{-1} \overset{p}\to P^{-1}$, $\frac{\partial S^{-1}}{\partial \eta_{k}}\overset{p}\to \frac{\partial P^{-1}}{\partial \eta_{k}}$, $\hat{\theta}^{\LS} \overset{p}\to \theta_{0}$, \eqref{eq:convergence in probability  of sqrtN_S_inv_P_inv}, \eqref{eq:convergence in probability of sqrtN_1st_derivative_S_inv_P_inv} and \eqref{eq:limiting in probability of sqrtN_PP_Sigma_inv} do not change with the value of $\eta$. It gives that as $N\to\infty$,
    \begin{align}
    \sqrt{N}\Upsilon_{1,y}|_{\eta_{y}^{*}} \overset{p}\to 0.
    \end{align}

\item For $\sqrt{N}\Upsilon_{1,y}|_{\eta_{y}^{*}}$, as $N\to\infty$, since $N(\Phi^{T}\Phi)^{-1}\overset{p}\to\Sigma^{-1}$, $S^{-1} \overset{p}\to P^{-1}$, $\frac{\partial S^{-1}}{\partial \eta_{k}}\overset{p}\to \frac{\partial P^{-1}}{\partial \eta_{k}}$, $\hat{\theta}^{\LS} \overset{p}\to \theta_{0}$ and \eqref{eq:limiting distribution of sqrtN_theta_LS_theta0} do not change with the value of $\eta$, we can derive that $\sqrt{N}\Upsilon_{2,y}|_{\eta_{y}^{*}}$ converges in distribution to a Gaussian distributed random variable with zero mean and the limiting variance is
\begin{align}
&4\sigma^{10}\left\{\theta_{0}^{T}\left[P^{-1}{\Sigma}^{-1}\frac{\partial P^{-1}}{\partial\eta_{k}} + \frac{\partial P^{-1}}{\partial\eta_{k}}{\Sigma}^{-1}P^{-1} \right]{\Sigma}^{-1}\right.\nonumber\\
&\left.\left.\left[P^{-1}{\Sigma}^{-1}\frac{\partial P^{-1}}{\partial\eta_{k}} + \frac{\partial P^{-1}}{\partial\eta_{k}}{\Sigma}^{-1}P^{-1} \right]\theta_{0}\right\}\right|_{\eta_{y}^{*}}
\end{align}
\end{enumerate}

Applying Slutsky's theorem, $\sqrt{N}\left.\frac{\partial \overline{\mathcal{F}_{\Sy}}}{\partial \eta}\right|_{\eta_{y}^{*}} $ converges in distribution to a Gaussian distributed random vector with zero mean and the $(k,l)$th element of the limiting covariance matrix is
\begin{align}\label{eq:limiting covariance of 1st derivative of Fsy}
&4\sigma^{10}\left\{\theta_{0}^{T}\left[P^{-1}{\Sigma}^{-1}\frac{\partial P^{-1}}{\partial\eta_{k}} + \frac{\partial P^{-1}}{\partial\eta_{k}}{\Sigma}^{-1}P^{-1} \right]{\Sigma}^{-1}\right.\nonumber\\
&\left.\left.\left[P^{-1}{\Sigma}^{-1}\frac{\partial P^{-1}}{\partial\eta_{l}} + \frac{\partial P^{-1}}{\partial\eta_{l}}{\Sigma}^{-1}P^{-1} \right]\theta_{0}\right\}\right|_{\eta^{*}_{y}}.
\end{align}

The second step is to show that the $(k,l)$th element of
\begin{align}
D_{y}(\eta^{*}_{y})\triangleq &\plim_{N\to\infty}N\E\left(\left.\frac{\partial \overline{\mathcal{F}_{\Sy}}}{\partial \eta}\right|_{\eta^{*}_{y}}\times \left.{\frac{\partial \overline{\mathcal{F}_{\Sy}}}{\partial \eta^{T}}}\right|_{\eta^{*}_{y}}\right)
\end{align}
equals \eqref{eq:limiting covariance of 1st derivative of Fsy}.

Using Lemma \ref{lemma:mean of quartic forms}, we have
\begin{align}
&D_{y}(\eta^{*}_{y})_{k,l}\nonumber\\
=&4\sigma^8\left\{\sigma^2\theta_{0}^{T}\left[P^{-1}{\Sigma}^{-1}\frac{\partial P^{-1}}{\partial\eta_{k}} + \frac{\partial P^{-1}}{\partial\eta_{k}}{\Sigma}^{-1}P^{-1} \right]{\Sigma}^{-1}\right.\nonumber\\
&\left[P^{-1}{\Sigma}^{-1}\frac{\partial P^{-1}}{\partial\eta_{l}} + \frac{\partial P^{-1}}{\partial\eta_{l}}{\Sigma}^{-1}P^{-1} \right]\theta_{0}\nonumber\\
&+\plim_{N\to\infty}\sqrt{N}\left[(\theta_{0})^{T}S^{-T}N(\Phi^{T}\Phi)^{-1}\frac{\partial S^{-1}}{\partial\eta_{k}}\theta_{0}\right.\nonumber\\
&-\left.\Tr\left(N(\Phi^{T}\Phi)^{-1}\frac{\partial S^{-1}}{\partial\eta_{k}}\right)\right]\nonumber\\
&\sqrt{N}\left[(\theta_{0})^{T}S^{-T}N(\Phi^{T}\Phi)^{-1}\frac{\partial S^{-1}}{\partial\eta_{l}}\theta_{0}\right.\nonumber\\
&-\left.\left.\left.\Tr\left(N(\Phi^{T}\Phi)^{-1}\frac{\partial S^{-1}}{\partial\eta_{l}}\right)\right]\right\}\right|_{\eta^{*}_{y}}.
\end{align}
Based on \eqref{eq:convergence in probability of sqrtN_1st_derivative_S_inv_P_inv} and \eqref{eq:limiting in probability of sqrtN_PP_Sigma_inv}, it can be derived that
\begin{align}
&\plim_{N\to\infty}\sqrt{N}\left[(\theta_{0})^{T}S^{-T}N(\Phi^{T}\Phi)^{-1}\frac{\partial S^{-1}}{\partial\eta_{k}}\theta_{0}\right.\nonumber\\
&-\Tr\left.\left.\left(N(\Phi^{T}\Phi)^{-1}\frac{\partial S^{-1}}{\partial\eta_{k}}\right)\right]\right|_{\eta_{y}^{*}}\nonumber\\
=&\plim_{N\to\infty}\sqrt{N}\left[(\theta_{0})^{T}S^{-T}N(\Phi^{T}\Phi)^{-1}\frac{\partial S^{-1}}{\partial\eta_{k}}\theta_{0}\right.\nonumber\\
&-\Tr\left(N(\Phi^{T}\Phi)^{-1}\frac{\partial S^{-1}}{\partial\eta_{k}}\right)-\left.\left.\frac{\partial W_{y}}{\partial\eta_{k}}\right]\right|_{\eta_{y}^{*}}\nonumber\\
=&\plim_{N\to\infty}\sqrt{N}\left\{\theta_{0}^{T}(S^{-T}-P^{-T})N(\Phi^{T}\Phi)^{-1}\frac{\partial S^{-1}}{\partial\eta_{k}}\theta_{0}\right.\nonumber\\
&+\theta_{0}^{T}P^{-T}[N(\Phi^{T}\Phi)^{-1}-\Sigma^{-1}]\frac{\partial S^{-1}}{\partial\eta_{k}}\theta_{0}\nonumber\\
&+\theta_{0}^{T}P^{-T}\Sigma^{-1}\left[\frac{\partial S^{-1}}{\partial\eta_{k}}- \frac{\partial P^{-1}}{\partial\eta_{k}}\right]\theta_{0}\nonumber\\
&-\Tr\left[(N(\Phi^{T}\Phi)^{-1}-\Sigma^{-1})\frac{\partial S^{-1}}{\partial\eta_{k}}\right]\nonumber\\
&-\Tr\left.\left.\left[\Sigma^{-1}\left(\frac{\partial S^{-1}}{\partial\eta_{k}}-\frac{\partial P^{-1}}{\partial\eta_{k}}\right)\right]
 \right\}\right|_{\eta_{y}^{*}}\nonumber\\
=&0.
\end{align}
Thus
\begin{align}
D_{y}(\eta^{*}_{y})_{k,l}
=&4\sigma^{10}\left\{\theta_{0}^{T}\left[P^{-1}{\Sigma}^{-1}\frac{\partial P^{-1}}{\partial\eta_{k}} + \frac{\partial P^{-1}}{\partial\eta_{k}}{\Sigma}^{-1}P^{-1} \right]{\Sigma}^{-1}\right.\nonumber\\
&\left.\left.\left[P^{-1}{\Sigma}^{-1}\frac{\partial P^{-1}}{\partial\eta_{l}} + \frac{\partial P^{-1}}{\partial\eta_{l}}{\Sigma}^{-1}P^{-1} \right]\theta_{0}\right\}\right|_{\eta^{*}_{y}}.
\end{align}
\end{itemize}

Finally, we apply Lemma \ref{lemma:asymptotic normality of a consistent root} with $M_{N}=N\overline{\mathcal{F}_{\Sy}}$ and $\eta^{*}=\eta^{*}_{y}$ and the proof of the asymptotic normality of $\hat{\eta}_{\Sy}-\eta_{y}^{*}$ is complete.

\subsection{Proof of Corollary \ref{corollary:ridge regression case for asymptotic normality}}

Inserting $P=\eta I_{n}$ into the first order optimality conditions of $W_{b}$ and $W_{y}$, we can derive that
\begin{align}
\eta_{b}^{*}=\frac{\theta_{0}^{T}\theta_{0}}{n},\
\eta_{y}^{*}=\frac{\theta_{0}^{T}\Sigma^{-1}\theta_{0}}{\Tr(\Sigma^{-1})}.
\end{align}
Inserting them into \eqref{eq:element of A eta_b_star}, \eqref{eq:element of B eta_b_star}, \eqref{eq:element of C eta_y_star} and \eqref{eq:element of D eta_y_star}, it gives that
\begin{align}
A_{b}(\eta_{b}^{*})=&\frac{n^3}{(\theta_{0}^{T}\theta_{0})^2},\\
B_{b}(\eta_{b}^{*})=&\frac{4\sigma^2n^4}{(\theta_{0}^{T}\theta_{0})^4}\theta_{0}^{T}\Sigma^{-1}\theta_{0},\\
C_{y}(\eta_{y}^{*})=&\frac{2\sigma^4\Tr^4(\Sigma^{-1})}{(\theta_{0}^{T}\Sigma^{-1}\theta_{0})^3},\\
D_{y}(\eta_{y}^{*})=&\frac{16\sigma^{10}\Tr^6(\Sigma^{-1})}{(\theta_{0}^{T}\Sigma^{-1}\theta_{0})^6}\theta_{0}^{T}\Sigma^{-3}\theta_{0},
\end{align}
which leads to
\begin{align}
\frac{A_{b}^{-1}(\eta_{b}^{*})B_{b}(\eta_{b}^{*})A_{b}^{-1}(\eta_{b}^{*})}{C_{y}^{-1}(\eta_{y}^{*})D_{y}(\eta_{y}^{*})C_{y}^{-1}(\eta_{y}^{*})}
=&\frac{\Tr^2(\Sigma^{-1})\theta_{0}^{T}\Sigma^{-1}\theta_{0}}{n^2\theta_{0}^{T}\Sigma^{-3}\theta_{0}}.
\end{align}
Apply the singular value decomposition (SVD) in $\Sigma$ as
\begin{align}
\Sigma=U_{s}S_{s}U_{s}^{T},
\end{align}
where $U_{s}\in\R^{n\times n}$ is orthogonal and $S_{s}\in\R^{n\times n}$ is diagonal with, eigenvalues of $\Sigma$, $\lambda_{1}(\Sigma)\geq \cdots \geq \lambda_{n}(\Sigma)$. Set $U_{s}^{T}\theta_{0}\triangleq \left[\begin{array}{ccc}\tilde{g}_{1} & \cdots & \tilde{g}_{n} \end{array}\right]^{T}$. As $\lambda_{n}(\Sigma)\to 0$ and the other eigenvalues are fixed, namely that $\cond(\Sigma)=\lambda_{1}(\Sigma)/\lambda_{n}(\Sigma)\to\infty$, we have
\begin{align}
&\frac{A_{b}^{-1}(\eta_{b}^{*})B_{b}(\eta_{b}^{*})A_{b}^{-1}(\eta_{b}^{*})}{C_{y}^{-1}(\eta_{y}^{*})D_{y}(\eta_{y}^{*})C_{y}^{-1}(\eta_{y}^{*})}\nonumber\\
=&\frac{1}{n^2}\frac{\frac{1}{\lambda_{n}}\tilde{g}_{n}^2+ \sum_{i=1}^{n-1}\frac{1}{\lambda_{i}}\tilde{g}_{i}^2}
{ \frac{1/\lambda_{n}^3}{(\sum_{j=1}^{n}1/\lambda_{j})^2}\tilde{g}_{n}^2 + \sum_{i=1}^{n-1}\frac{1/\lambda_{i}^3}{(\sum_{j=1}^{n}1/\lambda_{j})^2}\tilde{g}_{i}^2 }\nonumber\\
=&\frac{1}{n^2}\frac{\tilde{g}_{n}^2+ \sum_{i=1}^{n-1}\frac{\lambda_{n}}{\lambda_{i}}\tilde{g}_{i}^2}
{ \frac{1}{(1+\sum_{j=1}^{n-1}\lambda_{n}/\lambda_{j})^2}\tilde{g}_{n}^2 + \sum_{i=1}^{n-1}\frac{\lambda_{n}^3/\lambda_{i}^3}{(1+\sum_{j=1}^{n-1}\lambda_{n}/\lambda_{j})^2}\tilde{g}_{i}^2 }\nonumber\\
\to&\frac{1}{n^2}.
\end{align}

\section*{Appendix B}

\setcounter{subsection}{0}
\subsection{Matrix Norm Inequalities }

\begin{lemma}\label{lemma:matrix norm inequality}(\cite{PP2012} Chapter 10.3 Page $61-62$)
	{\it For the symmetric $B\in\R^{m\times m}$ with its rank $r$ and $C\in\R^{m\times m}$, we have
		\begin{align}
		\|BC\|_{F}\leq&\|B\|_{F}\|C\|_{F}\\
		\|B+C\|_{F}\leq&\|B\|_{F}+\|C\|_{F}\\
		|\Tr(B)|\leq&\|B\|_{*}\leq\sqrt{r}\|B\|_{F},
		\end{align}
where $\|\cdot\|_{*}$ denotes the nuclear norm. $|\Tr(B)|\leq\|B\|_{*}$ can be proved by $|\sum_{i=1}^{m}\lambda_{i}(B)|\leq \sum_{i=1}^{m}|\lambda_{i}(B)|$.
	}
\end{lemma}

\subsection{Strong Law of Large Numbers}

\begin{lemma}\label{lemma:SLLN}{ Kolmogorov's Strong Law of Large Numbers}(\cite{Greene2003} Page $1166$ Appendix D.7)
{\it	If $x_{i},i=1,\cdots,N$ is a sequence of independent random variables such that $\E(x_{i})=\mu_{i}<\infty$ and $\Var(x_{i})=\sigma^{2}_{i}<\infty$ such that $\sum_{i=1}^{\infty}\sigma^{2}_{i}/i^{2}<\infty$ as $N\to\infty$ then
\begin{align}
 \frac{1}{N}\sum_{i=1}^{N}x_{i}-\frac{1}{N}\sum_{i=1}^{N}\mu_{i} \overset{a.s.}\to 0.
\end{align}}
\end{lemma}

\subsection{Almost Sure Convergence of Sample Covariance Matrix}

\begin{corollary}\label{corollary:convergence of sample covariance matrix}
{\it Let $X_{1},\ X_{2},\ \cdots\ X_{N}$ be independent, identically distributed random vectors with mean $\mu_{x}$ and covariance matrix $\Sigma_{x}$, where $X_{i}\in\R^{n}$ for each $i=1,\cdots,N$, $\mu_{x}\in\R^{n}$ with $\|\mu_{x}\|_{2}<\infty$, and $\Sigma_{x}\in\R^{n\times n}$ with $\|\Sigma_{x}\|_{F}<\infty$. Then it can be seen that as $N\to\infty$,
\begin{align}\label{eq:as of mean vector}
\frac{\sum_{i=1}^{N}X_{i}}{N} \overset{a.s.}\to& \mu_{x}\\\label{eq:as of covariance matrix}
\frac{\sum_{i=1}^{N}(X_{i}-\bar{X})(X_{i}-\bar{X})^{T}}{N} \overset{a.s.}\to& \Sigma_{x},
\end{align}
\begin{align}
\frac{\sum_{i=1}^{N}X_{i}}{N} \overset{a.s.}\to \mu_{x},\
\frac{\sum_{i=1}^{N}(X_{i}-\bar{X})(X_{i}-\bar{X})^{T}}{N} \overset{a.s.}\to \Sigma_{x},
\end{align}
where $\bar{X}=\sum_{i=1}^{N}X_{i}/N$.}
\end{corollary}

\begin{proof}
Define that the $i$th element of $\mu_{x}$ is $\mu_{i}$, the $(i,i)$th element of $\Sigma_{x}$ is $\sigma^2_{i}$ and the $(i,j)$th ($i\not=j$) element of $\Sigma_{x}$ is $c_{i,j}$. Let $X_{i,j}$ represent the $j$th element of $X_{i}$.

According to Lemma \ref{lemma:SLLN}, since $\{X_{i,j}\}_{i=1}^{N}$ are i.i.d. with $\mu_{j}<\infty$ and $\sigma^{2}_{j}<\infty$, then it is clear that as $N\to\infty$, $\bar{X}_{j}\triangleq\sum_{i=1}^{N}X_{i,j}/N\overset{a.s.}\to\mu_{j}$. Meanwhile, as $N\to\infty$,
\begin{align}
\frac{1}{N}\sum_{i=1}^{N}(X_{i,j}-\bar{X}_{j})^{2}
=&\frac{1}{N}\sum_{i=1}^{N}X_{i,j}^2-\left(\frac{1}{N}\sum_{j=1}^{N}X_{i,j}\right)^2\nonumber\\
\overset{a.s.}\to & (\sigma_{j}^2+\mu_{j}^2)-\mu_{j}^2=\sigma_{j}^2.
\end{align}
In addition, for each pair $(j,k)$ ($j\not=k$), we can see that
\begin{align}
&\frac{1}{N}\sum_{i=1}^{N}(X_{i,j}-\bar{X}_{j})(X_{i,k}-\bar{X}_{k})\nonumber\\
=&\frac{1}{N}\sum_{i=1}^{N}X_{i,j}X_{i,k}-\left(\frac{1}{N}\sum_{i=1}^{N}X_{i,j}\right)
 \left(\frac{1}{N}\sum_{i=1}^{N}X_{i,k}\right)\nonumber\\
 \overset{a.s.}\to& (c_{j,k}+\mu_{j}\mu_{k})-\mu_{j}\mu_{k}=c_{j,k}.
\end{align}
Applying the results to respective elements of $X_{i},i=1,\cdots,N$, $\mu_{x}$ and $\Sigma_{x}$, then the results \eqref{eq:as of mean vector} and \eqref{eq:as of covariance matrix} can be obtained.
\end{proof}

\subsection{Upper Bound of the Frobenius Norm of A Random Matrix}

\begin{corollary}\label{corollary:upper bounds of building blocks}
	{\it For a positive definite random matrix $A_{N}\in\R^{n\times n}$, if $A_{N}=O_{p}(a_{N})$, the upper bound of $A_{N}^{-1}=O_{p}(1/a_{N})$ can be written as
	\begin{align}\label{eq:upper bound of inv Phiphi}
\|A_{N}^{-1}\|_{F}\leq \frac{1}{a_{N}}\frac{\sqrt{n}a_{N}}{\lambda_{1}(A_{N})}\cond(A_{N}).
	\end{align}}
\end{corollary}

\begin{proof}
According to the definition of Frobenius norm, we can know that
\begin{align}	
\|A_{N}^{-1}\|_{F}=&\sqrt{\sum_{i=1}^{n}\frac{1}{\lambda_{i}^2(A_{N})}}	
\leq\frac{1}{a_{N}}\frac{\sqrt{n}a_{N}}{\lambda_{1}(A_{N})}\cond(A_{N}).
\end{align}
\end{proof}

\subsection{Expectation and Covariance of Gaussian Quadratic Forms}

\begin{lemma}\label{lemma:mean of quartic forms}(\cite{RS2008} Chapter 5.2 Page $107-110$, \cite{PP2012} Chapter $8.2$ Page $43$)
{\it    Assume that $A\in\R^{n\times n}$ and $B\in\R^{n\times n}$. If $a\in\R^{n}$ follows the normal distribution with mean $\mu_{a}\in\R^{n}$ and the covariance matrix $\Sigma_{a}\in\R^{n\times n}$, i.e. $a\sim\mathcal{N}(\mu_{a},\Sigma_{a})$, then
    \begin{align}
    \E(a^{T}Aa)=&\Tr(A\Sigma_{a})+\mu_{a}^{T}A\mu_{a}\\
    \E(a^{T}Aaa^{T}Ba)=&\Tr(A\Sigma_{a}(B+B^{T})\Sigma_{a})\nonumber\\
                       &+\mu_{a}^{T}(A+A^{T})\Sigma_{a}(B+B^{T})\mu_{a}\nonumber\\
                       &+[\Tr(A\Sigma_{a})+\mu_{a}^{T}A\mu_{a}][\Tr(B\Sigma_{a})+\mu_{a}^{T}B\mu_{a}].
    \end{align}}
\end{lemma}

\subsection{Bounded in Probability and Convergence in Probability}

\begin{lemma}\label{lemma:Op and convergence in probability} (\cite{Janson2009}, Page 5, Lemma 3)
{\it If $X_{N}=O_{p}(a_{N})$ with $a_{N}$ to be positive number sequence, then for any positive number sequence $w_{N}$ which satisfies that as $N\to\infty$, $w_{N}\to\infty$, we have $X_{N}/w_{N}a_{N}\overset{p}\to 0$. Here $\overset{p}\to$ denotes the convergence in probability.}
\end{lemma}

\begin{proof}
$X_{N}=O_{p}(a_{N})$ means that $\forall\epsilon_{1}>0$, $\exists L>0$, such that
\begin{align}
&\text{P}(|X_{N}|>a_{N}L)<\epsilon_{1},\\
\Rightarrow& {\overline{\lim}}_{N\to\infty} \text{P}(|X_{N}|>a_{N}L)\leq\epsilon_{1},
\end{align}
where ${\overline{\lim}}_{N\to\infty}X_{N}=\lim_{k\to\infty}\sup_{N\geq k}X_{N}$ denotes the superior limit of the sequence.
Since as $N\to\infty$, we have $w_{N}\to\infty$, which also leads to $\epsilon_{2} w_{N}\to\infty$ for any $\epsilon_{2}>0$. Then for sufficiently large $N$, we always have $\epsilon_{2} w_{N}>L$, i.e.
\begin{align}
{\overline{\lim}}_{N\to\infty}\text{P}(|X_{N}|>\epsilon_{2} w_{N}a_{N})\leq {\overline{\lim}}_{N\to\infty}\text{P}(|X_{N}|>La_{N})\leq\epsilon_{1}.
\end{align}
Thus for any fixed $\epsilon_{2}>0$, $\forall \epsilon_{1}>0$, we have
\begin{align}
{\overline{\lim}}_{N\to\infty}\text{P}(|X_{N}|/w_{N}a_{N}>\epsilon_{2})\leq\epsilon_{1}.
\end{align}
If follows that $\forall \epsilon_{2}>0$, we have
\begin{align}
\lim_{N\to\infty}\text{P}(|X_{N}|/w_{N}a_{N}>\epsilon_{2})=0.
\end{align}
Therefore, $X_{N}/w_{N}a_{N}\overset{p}\to 0$ as $N\to\infty$.
\end{proof}

\subsection{Asymptotic Normality of A Consistent Root}

\begin{lemma}\label{lemma:asymptotic normality of a consistent root}(\cite{Amemiya1985} Theorem $4.1.3$ Page $111-112$)
{\it Make the assumptions:
\begin{enumerate}
\item \label{asp1 in lemma:Euclidean space}
Let $\Omega$ be an open subset of the Euclidean p-space. (Thus the true value $\eta^{*}$ is an interior point of $\Omega$. )
\item \label{asp2 in lemma:measurability of F}
$M_{N}(Y,\eta)$ is a measurable function of $Y$ for all $\eta\in\Omega$, and $\partial M_{N}/\partial\eta$ exists and is continuous in an open neighborhood $\Omega_{1}(\eta^{*})$ of $\eta^{*}$.
\item \label{asp3 in lemma:convergence of F}
There exists an open neighborhood $\Omega_{2}(\eta^{*})$ of $\eta^{*}$ such that $N^{-1}M_{N}(\eta)$ converges to a nonstochastic function $M(\eta)$ in probability and uniformly in $\eta$ in $\Omega_{2}(\eta^{*})$, and $M(\eta)$ attains a strict local minimum at $\eta^{*}$.
\item \label{asp4 in lemma:continuity of 2nd derivative}
$\partial^2 M_{N}/\partial\eta\partial\eta^{T}$ exists and is continuous in an open, convex neighborhood of $\eta^{*}$.
\item \label{asp5 in lemma:convergence of 2nd derivative}
$N^{-1}\left.(\partial^2 M_{N}/\partial\eta\partial\eta^{T})\right|_{\tilde{\eta}_{N}}$ converges to finite invertible $A(\eta^{*})=\lim_{N\to\infty}\E[ \left. N^{-1}(\partial^2 M_{N}/\partial\eta\partial\eta^{T})]\right|_{\eta^{*}}$ in probability for any sequence $\tilde{\eta}_{N}$ such that $\lim_{N\to\infty} \tilde{\eta}_{N}=\eta^{*}$ in probability.
\item \label{asp6 in lemma:convergence of 1st derivative}
$N^{-1/2}(\partial M_{N}/\partial \eta)|_{\eta^{*}}\overset{d}\to\mathcal{N}({\0},B(\eta^{*}))$, where $B(\eta^{*})=\lim_{N\to\infty}\left.\E[ N^{-1}(\partial M_{N}/\partial\eta)_{\eta^{*}}\times (\partial M_{N}/\partial\eta^{T})]\right|_{\eta^{*}}$ in probability.
\end{enumerate}
Let $\eta_{N}$ be the set of roots of the equation
\begin{align}
\frac{\partial M_{N}}{\partial \eta}=\0
\end{align}
corresponding to the local minima. Let $\{\hat{\eta}_{N}\}$ be a sequence obtained by choosing one element from ${\eta}_{N}$ such that $\plim \hat{\eta}_{N}=\eta^{*}$, where $\hat{\eta}_{N}$ can be called a consistent root. Then as $N\to\infty$,
\begin{align}
\sqrt{N}(\hat{\eta}_{N}-\eta^{*})\overset{d}\to\mathcal{N}({0},A(\eta^{*})^{-1}B(\eta^{*})A(\eta^{*})^{-1}).
\end{align}}
\end{lemma}

\subsection{Convergence in Probability}

\begin{lemma}\label{lemma:convergence in probability}(\cite{Amemiya1985} Theorem $4.1.5$ Page $113$)
{\it Suppose $M_{N}(\eta)$ converges in probability to a nonstochastic function $M(\eta)$ uniformly in $\eta$ in an open neighborhood of $\eta^{*}$. Then $\plim_{N\to\infty}M_{N}(\hat{\eta})=M(\eta^{*})$ if $\plim_{N\to\infty}\hat{\eta}=\eta^{*}$ and $M(\eta)$ is continuous at $\eta^{*}$.}
\end{lemma}

\addtolength{\textheight}{-12cm}   



\bibliographystyle{IEEEtran}
\bibliography{IEEEabrv,root}

\end{document}